\documentclass[final, a4paper, 12pt]{article}

\usepackage{amsfonts}
\usepackage{amsmath}
\usepackage{amssymb}
\usepackage{amscd}
\usepackage{amsthm}

\usepackage{bm}
\usepackage{cite}
\usepackage{showkeys}
\usepackage[english]{babel}
\usepackage{dblaccnt}
\usepackage{accents}
\usepackage{graphicx}

\usepackage{eucal}
\usepackage{color}
\usepackage{fancyhdr}
\usepackage{dsfont}
\usepackage[a4paper,scale=0.752]{geometry}
\usepackage{hyperref}

\newtheorem{definition}{Definition}
\newtheorem{lemma}[definition]{Lemma}
\newtheorem{theorem}[definition]{Theorem}
\newtheorem{corollary}[definition]{Corollary}

\newtheorem{example}[definition]{Example}

\newtheorem{remark}[definition]{Remark}

\newcommand*{\N}{\ensuremath{\mathbb{N}}}
\newcommand*{\Z}{\ensuremath{\mathbb{Z}}}
\newcommand*{\R}{\ensuremath{\mathbb{R}}}
\newcommand*{\C}{\ensuremath{\mathbb{C}}}
\renewcommand{\i}{\mathrm{i}}
\renewcommand{\phi}{\varphi}
\renewcommand{\d}[1]{\,\mathrm{d}#1 \,}
\newcommand{\dS}{\,\mathrm{dS} \,}
\newcommand{\D}{\mathcal{D}}
\newcommand{\ol}[1]{\overline{#1}}
\newcommand{\J}{\mathcal{J}} 

\renewcommand{\H}{{\mathcal{H}}}
\renewcommand{\S}{\mathbb{S}}
\renewcommand{\Re}{\mathrm{Re}\,}
\renewcommand{\Im}{\mathrm{Im}\,}

\newcommand{\divSpace}{\mathrm{div}\,}

\renewcommand{\tilde}[1]{{\widetilde{#1}}}
\newcommand{\LambdaAst}{{\Lambda^{\hspace*{-1pt}\ast}}}
\newcommand{\W}{{W_{\hspace*{-1pt}\Lambda}}} 
\newcommand{\Wast}{{W_{\hspace*{-1pt}\LambdaAst}}} 
\newcommand{\talpha}{{\tilde{\bm{\alpha}}}} 
\newcommand{\tx}{{\tilde{\bm{x}}}} 
\newcommand{\x}{{{\bm{x}}}} 
\renewcommand{\j}{{{\bm{j}}}} 
\newcommand{\z}{{\bm{z}}} 
\newcommand{\0}{{\bm{0}}} 
\newcommand{\y}{{\bm{y}}} 
\newcommand{\bell}{{\bm{\ell}}} 
\newcommand{\e}{{\mathbf{e}}}

\newcommand{\HH}{{\mathcal{H}}}
\newcommand{\WW}{{\mathcal{W}}}  
 
\newcommand{\s}{{\mathrm{s}}}

\renewcommand{\rho}{{\varrho}} 
\renewcommand{\epsilon}{{\varepsilon}}
\newcommand{\loc}{{\mathrm{loc}}}

\newlength{\dhatheight}
\newcommand{\dhat}[1]{%
    \settoheight{\dhatheight}{\ensuremath{\hat{#1}}}%
    \addtolength{\dhatheight}{-0.35ex}%
    \hat{\vphantom{\rule{1pt}{\dhatheight}}%
    \smash{\hat{#1}}}}

\begin{document}

\sloppy

\title{The Floquet-Bloch Transform and Scattering from Locally Perturbed Periodic Surfaces}
\author{Armin Lechleiter\thanks{Center for Industrial Mathematics, University of Bremen, Bremen, Germany; \texttt{lechleiter@math.uni-bremen.de}}}

\maketitle

\begin{abstract}
We use the Floquet-Bloch transform to reduce variational formulations of surface scattering problems for the Helmholtz equation from periodic and locally perturbed periodic surfaces to equivalent variational problems formulated on bounded domains. 
To this end, we establish various mapping properties of that transform between suitable weighted Sobolev spaces on periodic strip-like domains and coupled families of quasiperiodic Sobolev spaces. 
Our analysis shows in particular that the decay of solutions to surface scattering problems  from locally perturbed periodic surfaces is precisely characterized by the smoothness of its Bloch transform in the quasiperiodicity. 
\end{abstract}

\section{Introduction}
We analyze time-harmonic surface scattering modeled by the scalar Helmholtz equation from a periodic or a locally perturbed periodic surface $\Gamma$.  
A fundamental motivation to study surface scattering problems involving periodicity is the growing industrial importance of micro or nano-structured surfaces in optics, requiring ever more accurate models and simulations. 
Neglecting periodicity, such scattering problems can be considered as rather particular rough surface scattering problems and tackled by variational or integral equation formulations~\cite{Chand1999, Chand2005, Chand2006a, Chand2010}. 
Unfortunately, these formulations are posed on unbounded domains such that, e.g., convergence analysis for any numerical discretization automatically is non-standard.

Motivated by the recent paper~\cite{Coatl2012} which treats the Helmholtz equation in a periodic medium with a line defect by the (Floquet-)Bloch transform, we show in this paper that an analogous partial transform can be used to transform scattering problems for unbounded periodic surfaces with local perturbations to equivalent problems on bounded domains. 
Any variational formulation of such a transformed problems hence possesses the advantage of straightforward discretization by standard techniques.
Note that~\cite{Hadda2015} analyzes the discretization of such a problem in two dimensions in a setting that is somewhat easier due to absorption, and that~\cite{Fliss2016} uses the Bloch transform to study scattering in an infinite wave guide. 

The Floquet-Bloch transform can be interpreted as, roughly speaking, a sort of periodic Fourier transform, see~\cite{Reed1987,Kuchm1993}.
Despite mapping properties of that transform are in principle known, in particular its isometry property on $L^2$-spaces, we prove that a certain partial Bloch transform is an isomorphism between weighted Sobolev spaces on periodic surfaces or periodic strips and coupled families of Sobolev spaces of quasiperiodic functions.
(Suitable references seem to be lacking.)
Even if these mapping properties are independent of dimension, we restrict ourselves to surfaces and domains of dimension two and three, respectively, see Remark~\ref{rem:1}. 
(See further~\cite[Annexe B]{Fliss2009} for corresponding results in the one-dimensional case.)

The partial Bloch transform allows to equivalently reformulate scattering problems from periodic or locally perturbed periodic surfaces as a (coupled) family of quasiperiodic scattering problems on, roughly speaking, a unit cell of the periodic domain. 
Apart from the Bloch transform, our main technique for treating perturbed periodic surfaces is a suitable diffeomorphism between the perturbed periodic domain and the unperturbed one. 
We exploit this equivalence to analyze solutions to scattering problems for particular incident Herglotz wave functions that may serve, e.g., as models for incident Gaussian beams. 
For simplicity, we restrict ourselves to Robin or impedance boundary conditions involving a periodic coefficient, and merely comment on other boundary conditions. 

Our results in particular show that the decay of the solution to a scattering problem from a (perturbed) periodic surface is characterized by smoothness of the solution to the transformed problem in the quasiperiodicity.
Decay results for solutions to rough surface scattering problems with Dirichlet boundary condition have been established previously in~\cite{Chand2010}; some of these results are, roughly speaking, validated by our findings for different boundary conditions (see Remark~\ref{th:chandler}).  
We further indicate that for a particular class of incident Herglotz wave functions, the wave field scattered from a (locally perturbed) periodic surface decays more rapidly than the bounds from~\cite{Chand2010} would suggest. 

The remainder of this paper is structured as follows: 
After presenting the Bloch transform in $\R^2$ in Section~\ref{se:bloch}, we introduce function spaces to analyze this transform in Section~\ref{se:functionSpaces}. 
Next, Sections~\ref{se:propBloch},~\ref{se:blochSurf}, and~\ref{se:blochDom} show properties of the Bloch transform on $\R^2$, on (bi-)periodic surfaces, and on periodic domains, respectively. 
Section~\ref{se:perioSurfScatt} introduces surface scattering problems of incident acoustic waves from periodic surfaces and tackles those by the Bloch transform. 
Section~\ref{se:regDec} extends these results by regularity estimates for the solution in the quasiperiodicity parameter. 
Section~\ref{se:pertScatt} finally tackles similar problems for perturbed periodic surface scattering via suitable diffeomorphisms.

\textit{Notation:} 
We write $\tx = (\x_1,\x_2)^\top$ for points $\x=(\x_1,\x_2,\x_3)^\top$ in $\R^3$ and set, in analogy, $\tilde\nabla f= (\partial_{\x_1} f, \partial_{\x_2} f)^\top$, whereas $\nabla f = (\partial_{\x_1} f, \partial_{\x_2} f, \partial_{\x_3} f)^\top$ is the gradient of $f$. 
By $\e_{(1,2,3)}$ we denote the standard basis vectors of $\R^3$.
All function spaces we consider generically contain complex-valued functions. 
The space of smooth functions in a domain $U$ with smooth extension to the boundary up to arbitrarily high order is $C^\infty(\overline{U})$. 
Constants $C$ and $c$ are generic and might change from line to line. 
We further use the symbol $\simeq$ to indicate equivalence of two expressions up to positive constants.

\section{The Bloch Transform in $\R^2$}\label{se:bloch}
In the entire paper we fix an invertible matrix $\Lambda\in \R^{2\times 2}$ and call a function or a vector field $\phi: \, \R^2 \to \C^d$ $\Lambda$-periodic if $\phi(\x+\Lambda j) = \phi(\x)$ for all $\x \in \R^2$ and all $j\in\Z^2$. 
It satisfies to require the latter condition in the fundamental domain of periodicity for the lattice $\{ \Lambda \j: \, \j\in\Z^2 \} \subset \R^2$, the so-called Wigner-Seitz-cell,
\[
  \W := \big\{ \Lambda \tilde{z}: \, \tilde{\z} \in \R^2, \, -1/2< \tilde{\z}_{1,2} \leq 1/2 \big\} \subset \R^2. 
\]
The dual periodicity matrix then equals $\LambdaAst := 2\pi \Lambda^{-\top}= 2\pi(\Lambda^\top)^{-1} \in \R^{2\times 2}$ and defines the dual (or reciprocal) lattice $\{ \LambdaAst \bell: \, \bell\in\Z^2 \}$, as well as the dual fundamental domain of periodicity, the so-called Brillouin zone  
\[
  \Wast := \big\{ \LambdaAst \tilde{\bm{\mu}}: \, \tilde{\bm{\mu}}\in \R^2, \, -1/2< \tilde{\bm{\mu}}_{1,2} \leq 1/2 \big\} \subset \R^2. 
\]
Periodicity with respect to the lattice defined via $\LambdaAst$ is of course defined analogously. 
To generalize the notion of periodicity, we further define for arbitrary $\talpha \in \R^2$ that $\phi: \, \R^2 \to\C^d$ is $\talpha$-quasiperiodic with respect to $\Lambda$ if
\begin{equation}
  \label{eq:LQuasiperiodicity}
  \phi(\x + \Lambda \j) = e^{\i \talpha\cdot \Lambda \j} \, \phi(\x), 
  \quad 
  \x\in \R^2.
\end{equation}
A $\Lambda$-periodic function becomes hence $\talpha$-quasiperiodic by multiplication with $\x \mapsto \exp(\i \talpha \cdot \x)$.
In case that $\talpha=\LambdaAst \bell$ for some $\bell\in\Z^2$, one easily computes that $\talpha$-quasiperiodicity  boils down to periodicity.
Thus, $\talpha$-quasiperiodicity is the same as $(\talpha+\LambdaAst \bell)$-quasiperiodicity and it is sufficient to consider quasiperiodicities in $\Wast$. 

The fundamental tool of our analysis is the two-dimensional Bloch transform $\J_{\R^2}$, that we define for $\phi \in C^\infty_0(\R^2)$ by  
\begin{equation}
  \label{eq:BlochZeta}
  \J_{\R^2} \phi(\talpha, \tx) 
  := \frac{|\det \Lambda|}{2\pi}^{1/2} \sum_{\j\in \Z^2} \phi(\tx + \Lambda \j) \, e^{-\i \talpha \cdot \Lambda \j},
  \quad  \talpha \in \R^2, \, \tx \in \R^2. 
\end{equation}
This transform is well-defined as $\phi$ is continuous and has compact support. 
We can actually restrict the second argument $\tx $ of $\J_{\R^2} \phi$ to $\W$ since the right-hand side of~\eqref{eq:BlochZeta} to $\R^2$ defines an $\talpha$-quasiperiodic function in $\tx$ with respect to $\Lambda$, 
\begin{multline*}
  \J_{\R^2} \phi (\talpha, \tx +\Lambda \j')
  = \frac{|\det \Lambda|}{2\pi}^{1/2} \sum_{\j\in \Z^2} \phi(\tx + \Lambda (\j+\j')) \, e^{-\i \talpha \cdot \Lambda \j}\\
  = \frac{|\det \Lambda|}{2\pi}^{1/2} \sum_{\j\in \Z^2} \phi(\tx + \Lambda \j) \, e^{-\i \talpha \cdot \Lambda (\j-\j')}
  = e^{\i \talpha \cdot \Lambda \j'} \,  \J_{\R^2} \phi(\talpha, \tx), 
  \quad  \tx \in \R^2, \, \j' \in \Z^2,
\end{multline*}
such that the knowledge of $\tx \mapsto \J_{\R^2} \phi(\talpha, \tx)$ in $\W$ defines that function everywhere in $\R^2$. 
For fixed $\tx $, the right-hand side in~\eqref{eq:BlochZeta} moreover is a Fourier series in $\talpha$, such that $\talpha \mapsto \J_{\R^2} \phi(\talpha, \tx)$ is $\LambdaAst$-periodic. 
Thus, we can as well restrict the quasiperiodicity $\talpha$, i.e., the first argument of $\J_{\R^2} \phi$, to $\Wast$ without loosing information.  
Let us further note that $\J_{\R^2}$ commutes with $\Lambda$-periodic functions: If $q: \, \R^2 \to \C$ is $\Lambda$-periodic, then 
\begin{equation}\label{eq:JComPerio}
  [\J_{\R^2} (q\phi)](\talpha, \tx) 
  = \frac{|\det \Lambda|}{2\pi}^{1/2} \sum_{\j\in\Z^2} q(\tx + \Lambda \j) \phi(\tx + \Lambda \j) e^{-\i\talpha\cdot \Lambda \j}
  = q(\tx) \, \J_{\R^2} \phi (\talpha, \tx).
\end{equation}

\begin{remark}\label{rem:1}
As we do not exploit specific two-dimensional features, all established results easily extend to higher dimensions.
More precisely, in arbitrary dimension $d \geq 1$ the only required modification is to change the normalizing factor of the Bloch transform in~\eqref{eq:BlochZeta} to $|\det\Lambda|^{1/2}/ (2\pi)^{d/2}$.
We omit this straightforward generalization to simplify notation. 
\end{remark} 
%

\section{Adapted function spaces for the Bloch transform}\label{se:functionSpaces}
To analyze the Bloch transform, we need to introduce the usual weighted Sobolev spaces as well as spaces of (quasi-)periodic functions, see, e.g.,~\cite{Lions1972, Bergh1976, McLea2000}. 
To this end, a convenient tool is the Fourier transform, 
\begin{equation}
  \label{eq:fourierTrafo}
  \dhat{\phi}(\z) := \frac{1}{2\pi} \int_{\R^2} e^{-\i \, \z \cdot \tx} \phi(\tx) \d{\tx} 
  \quad \text{for } \phi \in \C^\infty_0(\R^2) \text{ and } \z \in \R^2. 
\end{equation}
This transform extends to an isometry on $L^2(\R^2)$ and defines Bessel potential spaces via  
\[
  \HH^s(\R^2) := 
  \bigg\{ \phi \in \mathcal{D}'(\R^2): \, \int_{\R^2} (1+|\z|^2)^s |\dhat{\phi}(\z)|^2 \d{\z} <\infty \bigg\},
  \quad s \in \R,
\]
such that the Fourier transform of functions in $\HH^s(\R^2)$ times the weight $\z \mapsto (1+|\z|^2)^{s/2}$ belongs to $L^2(\R^2)$. 
Equipped with the norm $\| \phi \|_{\HH^s(\R^2)} = ( \int_{\R^2} (1+|\z|^2)^2 |\dhat{\phi}(\z)|^2 \d{\z} )^{1/2}$, the space $\HH^s(\R^2)$ is a Hilbert space. 
The corresponding weighted spaces are
\[
  \HH^s_r(\R^2) := \left\{ \phi \in \mathcal{D}'(\R^2): \, \tx \mapsto (1+|\tx|^2)^{r/2} \phi(\tx) \in \HH^s(\R^2) \right\},
  \quad s,r \in \R, 
\]
equipped with the norm $\| \phi \|_{\HH^s_r(\R^2)} = \| \tx \mapsto (1+|\tx|^2)^{r/2} \phi(\tx)  \|_{\HH^s(\R^2)}$.
(By the Leibniz formula, see~\eqref{eq:leibniz}, an equivalent norm in $\HH^s_r(\R^2)$ is $\| \phi \|_{\HH^s_r(\R^2)} = \| \z \mapsto (1+|\z|^2)^{s/2} \dhat{\phi}(\z)  \|_{\HH^r(\R^2)}$.)
If $s=m \in \N_0$ is an integer, then the latter norm is equivalent to a sum of weighted $L^2$-norms of weak derivatives, 
\begin{equation} \label{eq:HsmNorm}
  \| \phi \|_{\HH^s_m(\R^2)}
  \simeq 
  \Bigg[ \sum\limits_{\bm\gamma\in\N_0^2, |\bm\gamma|\leq m} \big\| \tx \mapsto \partial^{\bm\gamma} \big[ (1+|\tx|^2)^{r/2} \phi(\tx)\big] \big\|_{L^2(\R^2)}^2 \Bigg]^{1/2} . 
\end{equation}

We further recall the usual Sobolev spaces $\HH^s(\W)$ and $\HH^s(\Wast)$ on the bounded Lipschitz domains, defined, e.g., by interpolation and a duality argument via the spaces $\HH^m(\W)$ and $\HH^m(\Wast)$ for $m\in\N$, see~\cite{McLea2000}. 
To define corresponding spaces of periodic functions, we note that the smooth, $\Lambda$-periodic functions 
\begin{equation}\label{eq:LPerF}
  \phi_\Lambda^{(\j)}(\tx) := \frac{1}{| \det \Lambda|^{1/2}} e^{\i \LambdaAst \j \cdot \tx}, \quad \j\in\Z^2,
\end{equation}
form a complete orthonormal system in $L^2(\W)$. 
For any $\talpha \in \Wast$, the space $\mathcal{D}_{\talpha}'(\R^2)$ of $\talpha$-quasiperiodic distributions with respect to $\Lambda$ contains the products of all periodic distributions, see~\cite{Saran2002}, with $\tx \mapsto \exp(\i \talpha \cdot \tx)$. 
For such distributions and $\j \in \Z^2$, we define Fourier coefficients  
\begin{equation}
  \label{eq:fourierCoeff}
  \hat{\phi}(\j) := \big\langle\phi, \, \tx \mapsto \overline{\exp(\i \talpha \cdot \tx) \phi_\Lambda^{(\j)}(\tx)}\big \rangle_\Lambda
  \, \left[ =  \int_{\W} \phi(\tx) e^{-\i \talpha \cdot \tx} \overline{\phi_\Lambda^{(\j)}(\tx)} \d{\tx}
  \ \text{if } \phi \in L^2(\W) \right] 
\end{equation}
and introduce the subspace $\HH^s_\talpha(\W)$ of $\mathcal{D}'_\talpha(\R^2)$ containing all $\talpha$-quasiperiodic distributions with finite norm 
\begin{equation}
  \label{eq:HsDef}
  \| \phi \|_{\HH^s_\talpha(\W)} := \Bigg[ \sum_{\j \in \Z} (1+|\j|^2)^s \, |\hat{\phi}(\j)|^2 \Bigg]^{1/2} < \infty, \qquad s\in \R.
\end{equation}
When equipped with the norm from~\eqref{eq:HsDef}, this space becomes a Hilbert space with inner product $(\phi_1,\phi_2)_{\HH^s_\talpha(\W)} := \sum_{\j \in \Z} (1+|\j|^2)^s \, \hat{\phi}_1(\j) \overline{\hat{\phi}_2(\j)}$. 
It follows from basic Hilbert space theory that elements of $\HH^s_\talpha(\W)$ can be represented by their Fourier series, i.e., 
\[
  \phi(\tx) = \sum_{\j\in\Z^2} \hat{\phi}(\j) e^{\i \talpha \cdot \tx} \phi_\Lambda^{(\j)}(\tx) 
  = \frac{1}{| \det \Lambda|^{1/2}} \sum_{\j\in\Z^2} \hat{\phi}(\j) e^{\i (\LambdaAst \j+\talpha) \cdot \tx}   
  \quad \text{holds in $\HH^s_\talpha(\W)$.}
\]
This series representation in turn implies that for non-negative integers $s=m\in\N_0$, the space $\HH^s_\talpha(\W)$ consists of $m$ times weakly differentiable $\talpha$-quasiperiodic functions with weak derivatives in $L^2$,
\begin{equation}\label{eq:normEqual}
  \| \phi \|_{\WW^{2,m}_\talpha(\W)}
  := \Bigg[ \sum_{\bm\gamma\in\N_0^2, |\bm\gamma|\leq m} \big\| \partial^{\bm\gamma} \phi \big\|_{L^2(\W)}^2 \Bigg]^{1/2}
  \simeq 
  \| \phi \|_{\HH^m_\talpha(\W)}
\end{equation}

\begin{remark}
Fourier coefficients $\hat\phi(\ell)$ of $\LambdaAst$-periodic functions are defined as in~\eqref{eq:fourierCoeff} via the complete orthonormal system $\{ \phi_\LambdaAst^{(\ell)}: \, \ell\in\Z^2 \} \subset L^2(\Wast)$, where $\psi_\LambdaAst^{(\ell)}(\talpha) := | \det \LambdaAst|^{-1/2} \exp(\i  \, \Lambda \ell \cdot \talpha)$ for $\ell\in\Z^2$.  
We define corresponding spaces $\HH^s_{\bm{0}}(\Wast)$ as in~\eqref{eq:HsDef}. 
\end{remark} 

We have already noted above that the Bloch transform $\J_{\R^2} \phi$ extends to a $\LambdaAst$-periodic function in $\talpha$ and to a quasiperiodic function in $\tx $ with quasiperiodicity $\talpha$. 
It is natural that we require adapted function spaces in $(\talpha,\tx)$.  
To this end, we introduce the vector space $\mathcal{D}_{\Lambda}'(\R^2\times\R^2)$ of distributions in $\mathcal{D}'(\R^2 \times \R^2)$ that are $\LambdaAst$-periodic in their first and quasiperiodic with respect to $\Lambda$ in their second variable, with quasiperiodicity equal to the first variable.
The set of corresponding test functions is denoted as $C^\infty_\Lambda(\R^2 \times \R^2)$. 
Any $\psi \in \mathcal{D}_{\Lambda}'(\R^2\times\R^2)$ then possesses two Fourier series representations
\begin{align*}
  \psi(\talpha,\tx) 
  & = \sum_{\j\in\Z^2}  \hat\psi_\Lambda(\talpha,\j)  e^{\i\talpha\cdot \tx} \phi_{\Lambda}^{(\j)}(\tx)
  = \frac{1}{| \det \Lambda|^{1/2}} \sum_{\j\in\Z^2} \hat\psi_\Lambda(\talpha,\j) e^{\i (\LambdaAst \j+\talpha) \cdot \tx} \\
  & = \sum_{\bell\in\Z^2}  \hat\psi_\LambdaAst(\bell,\tx) \phi_\LambdaAst^{(\bell)} (\talpha)
  = \frac{1}{| \det \LambdaAst|^{1/2}} \sum_{\bell\in\Z^2} \hat\psi_\LambdaAst(\bell,\tx) e^{\i  \, \Lambda \bell \cdot \talpha} 
\end{align*}
with Fourier coefficients $\hat\psi_\Lambda(\talpha,\j) =\langle\psi(\talpha,\cdot),\phi_{\Lambda}^{(\j)} \rangle_\Lambda$ and $\hat\psi_\LambdaAst(\bell,\tx) = \langle\psi(\cdot,\tx),\phi_{\LambdaAst}^{(\bell)}\rangle_\LambdaAst$. 
For $s,r \in \R$, these coefficients yield Sobolev spaces $L^2(\Wast; \HH^s_\talpha(\W))$ and $\HH_\0^r(\Wast; L^2(\W))$ of distributions in $\D_\Lambda'(\R^2 \times \R^2)$ with finite norms 
\begin{align} \label{eq:HrHsPre1}
  \| \psi \|_{L^2(\Wast; \HH^s_\talpha(\W))} 
  &:= \bigg[ \sum_{\j\in\Z^2} (1+|\j|^2)^s \int_\Wast  | \hat\psi_\Lambda(\talpha,\j) |^2 \d{\talpha}\bigg]^{1/2} < \infty \qquad \text{and} \\
  \| \psi \|_{\HH_\0^r(\Wast; L^2(\W))}
  &:= \bigg[ \sum_{\bell\in\Z} (1+|\bell|^2)^r \int_\W | \hat\psi_\LambdaAst(\bell,\tx) |^2 \d{\tx} \bigg]^{1/2} < \infty, \label{eq:HrHsPre2}   
\end{align}
respectively. 
(The bold-face index $\0$ indicates periodicity and has nothing to do with boundary conditions!)
If $r=\ell \in \N$ is an integer, then it is well-known from Fourier series theory that the latter space contains all $\LambdaAst$-periodic distributions with values in $L^2(\W)$ that possess weak derivatives with respect to $\talpha$ belonging to $L^2(\Wast)$ up to order $\ell$. 
Moreover,
\begin{equation} \label{eq:equalNorm2}
\begin{split}
  \| \psi \|_{\WW^{2,\ell}_\0(\Wast; L^2(\W))}^2  
  & :=  \sum_{\bm\gamma\in\N_0^2, |\bm\gamma|\leq \ell} \int_\Wast \int_\W \left|\partial^{\bm\gamma}_\talpha \psi(\talpha, \tx) \right|^2 \d{\tx} \d{\talpha} 
  \simeq \| \psi \|_{\HH^\ell_\0(\Wast; L^2(\W))}^2
\end{split}
\end{equation}
are equivalent norms on $\HH^\ell_\0(\Wast;L^2(\W))$. 
Replacing $L^2(\W)$ by $\HH^s_\talpha(\W)$, we define $\WW^{2,\ell}_\0(\Wast; \HH^s_\talpha(\W))$ as space of all elements in $\mathcal{D}'_\Lambda(\R^2 \times \R^2)$ with finite norm 
\begin{equation} \label{eq:equalNorm3}
\begin{split}
  \| \psi \|_{\WW^{2,\ell}_\0(\Wast; \HH^s_\talpha(\W))}
  & := \Bigg[ \sum_{\bm\gamma\in\N_0^2, |\bm\gamma|\leq \ell} \int_\Wast \| \partial^{\bm\gamma}_\talpha \psi(\talpha, \cdot) \|_{\HH^s_\talpha(\W)}^2 \d{\talpha} \Bigg]^{1/2} < \infty.
\end{split}
\end{equation}
For $s=m\in\N$ an integer, the norm equivalence in~\eqref{eq:normEqual} shows that 
\begin{align*} \label{eq:HrHs}
  \| \psi \|_{\WW^{2,\ell}_\0(\Wast; \WW^{2,m}_\talpha(\W))}^2
  & := \sum_{\bm\gamma\in\N_0^2, |\bm\gamma|\leq \ell} \sum_{\bm\eta\in\N_0^2, |\bm\eta|\leq m} \int_\Wast \int_\W \left| \partial^{\bm\gamma}_\talpha \partial^{\bm\eta}_\tx \psi(\talpha, \tx) \right|^2 \d{\tx} \d{\talpha}\\
  & \simeq  \sum_{\bm\gamma\in\N_0^2, |\bm\gamma|\leq \ell} \int_\Wast \| \partial^{\bm\gamma}_\talpha \psi(\talpha, \cdot) \|_{\HH^m_\talpha(\W)}^2 \d{\talpha} 
  = \| \psi \|_{\WW^{2,\ell}_\0(\Wast; \HH^m_\talpha(\W))}^2 
\end{align*}
yields an equivalent (squared) norm in $\WW^{2,\ell}_\0(\Wast; \HH^m_\talpha(\W))$. 
Interpolation of the periodic function spaces $\WW^\ell_\0(\Wast; \HH^s_\talpha(\W))$ in the smoothness parameter $\ell$ then yields a family of periodic interpolation spaces included in $\mathcal{D}'_\Lambda(\R^2 \times \R^2)$ that we denote by 
\begin{equation}\label{eq:HrHs}
  \HH^{\ell+\theta}_\0(\Wast; \HH^s_\talpha(\W))
  := \left[ \HH^\ell_\0(\Wast; \HH^s_\talpha(\W)), \HH^{\ell+1}_\0(\Wast; \HH^s_\talpha(\W)) \right]_\theta 
  \quad \text{for $0<\theta<1$,} 
\end{equation}
see~\cite{Lions1972, Bergh1976}. 
The dual space of $\HH^r_\0(\Wast; \HH^s_\talpha(\W))$ with respect to the inner product 
\[
  (\psi_1,\psi_2) \mapsto 
  \int_\Wast \big(\psi_1(\talpha,\cdot),\psi_2(\talpha,\cdot)\big)_{\HH^s_\talpha(\W)} \d{\talpha}
\]
of $L^2(\Wast; \HH^s_\talpha(\W))$ is denoted by $\HH^{-r}_\0(\Wast; \HH^s_\talpha(\W))$, such that this class of spaces is well-defined for all $r,s \in \R$.
(We could have defined these spaces as well starting from $L^2(\Wast; \HH^s_\talpha(\W))$ from~\eqref{eq:HrHsPre1} instead of $\HH^s_\0(\Wast; L^2(\W))$ from~\eqref{eq:HrHsPre2}.)

\begin{remark} \label{rem:X}
(a) If $s=r=0$, all spaces introduced so-far reduce to $L^2(\W)$, $L^2(\Wast)$ or $L^2(\Wast; L^2(\W))$ due to Parseval's or Plancherel's equality.

(b) The construction of $\HH^r_\0(\Wast; \HH^s_\talpha(\W))$ generalizes to arbitrary families of Hilbert spaces $X_\talpha$ that contain distributions that are $\talpha$-quasiperiodic in the first two variables $\tx$. 
For $r \in \R$, the squared norm of an element $\psi$ in $\HH^m_\0(\Wast; X_\talpha)$ then equals 
\begin{equation} \label{eq:HrX}
  \| \psi \|_{\HH^r_\0(\Wast; X_\talpha)}^2  
  := \sum_{\bell\in\Z^2} (1+|\ell|^2)^r \big\| \x \mapsto \big\langle \psi(\cdot,x), \phi_\LambdaAst^{(\ell)} \big\rangle_\LambdaAst \big\|^2_{X_\talpha}.
\end{equation}
We later on use the resulting spaces for quasiperiodic Sobolev spaces $X_\talpha = \HH^s_\talpha(\Omega_H^\Lambda)$ on a three-dimensional domain $\Omega_H^\Lambda$, see~\eqref{eq:HmAlpha}. 
\end{remark}

\section{Properties of the Bloch transform on $\R^2$}\label{se:propBloch}
\begin{theorem}\label{th:BlochR}
(a) The Bloch transform $\J_{\R^2}$ extends from $C^\infty_0(\R^2)$ to an isometric isomorphism between $L^2(\R^2)$ and $L^2(\Wast; L^2(\W))$ with inverse
\begin{equation}  \label{eq:X1}
    \J_{\R^2}^{-1} \tilde{\phi} (\tx + \Lambda \j)
    = \frac{|\det \Lambda|}{2\pi}^{1/2} \int_{\Wast} \tilde{\phi}(\talpha, \tx) e^{\i \talpha \cdot \Lambda\j} \d{\talpha}, 
    \quad \tx \in \W, \, \j \in \Z^2.  
\end{equation}
(b) For $s$ and $r \in \R$, the Bloch transform $\J_{\R^2}$ extends from $C^\infty_0(\R^2)$ to an isomorphism between $\HH^s_r(\R^2)$ and $\HH^r_\0(\Wast; \HH^s_\talpha(\W))$. 
Its inverse transform is given by~\eqref{eq:X1} with equality in the sense of the norm of $\HH^s_r(\R^2)$. 
\end{theorem}
\begin{proof}
(a)
Since the Bloch transform $\psi = \J_{\R^2} \phi$ of $\phi \in C^\infty_0(\R^2)$ is an $\talpha$-quasiperiodic function in $\tx$, we can develop $\tx \mapsto \exp(-\i\talpha \cdot \tx) \, (\J_{\R^2} \phi)(\talpha,\tx)$ for fixed $\talpha$ into a Fourier series with coefficients
\begin{align*}
  c(\talpha, \j) 
  & = \int_\W \J_{\R^2} \phi(\tx) e^{-\i\talpha \cdot \tx} \overline{\phi_\Lambda^{(\j)}(\tx)} \d{\tx} 
  = \frac{1}{2\pi} \int_\W \sum_{\bell\in \Z^2} \phi(\tx + \Lambda \bell) e^{-\i \talpha \cdot \Lambda \bell} e^{-\i \LambdaAst \j \cdot \tx} \d{\tx} \\
  & \stackrel{\tx +\Lambda \bell =\tilde{\y}}{=}  
     \frac{1}{2\pi} \int_{\R^2} \phi(\tilde{\y}) e^{-\i (\LambdaAst \j + \talpha) \cdot \tilde{\y}}  \d{\tilde{\y}} 
  = \dhat{\phi}\left(\LambdaAst \j + \talpha \right), \quad \j \in \Z^2, \, \talpha \in \R^2, 
\end{align*}
where we exploited that $\exp(-\i \LambdaAst \j \cdot (\tilde{\y} - \Lambda \bell)) = \exp(-\i \LambdaAst \j \cdot \tilde{\y})$.
In consequence, 
\begin{equation}\label{eq:blochHoppla}
  \J_{\R^2} \phi(\talpha, \tx) 
  = \sum_{\j \in \Z^2} \dhat{\phi}\left( \LambdaAst \j+\talpha \right) e^{\i\talpha \cdot \tx} \varphi_\Lambda^{(\j)}(\tx)
  = |\det \Lambda|^{-1/2} \sum_{\j \in \Z^2} \dhat{\phi}\left( \LambdaAst \j+\talpha \right) e^{\i (\LambdaAst\j +\talpha) \cdot \tx},
\end{equation}
and Parseval's and Plancherel's equality imply that 
\[
  \|  \J_{\R^2} \phi \|_{L^2(\Wast; L^2(\W))}^2 
   = \sum_{\j \in \Z^2} \int_\Wast \left| \dhat{\phi}\left( \LambdaAst \j+\talpha \right) \right|^2 \d{\talpha}
   = \int_{\R^2} \big| \dhat{\phi}\left( \tilde{\bm{\mu}} \right)\big|^2 \d{\tilde{\bm{\mu}}} 
   = \| \phi \|_{L^2(\R^2)}^2,
\]
such that $\J_{\R^2}$ extends from $C^\infty_0(\R^2)$ to an isometry from $L^2(\R^2)$ into $L^2(\Wast; L^2(\W))$.
Since the Fourier transform is moreover an isomorphism on $L^2(\R^2)$, the Bloch transform is an isomorphism, too.  
Due to~\eqref{eq:blochHoppla}, the Fourier inversion formula further implies for $\phi \in C^\infty_0(\R^2)$ and $\tx \in \R^2$ that  
\begin{align*}
  \J_{\R^2}^{-1} \big( \J_{\R^2} \phi \big) (\tx)
  & = \frac{1}{2\pi} \int_{\Wast} \sum_{\j \in \Z^2} 
    \dhat{\phi}\left( \LambdaAst \j + \talpha \right) e^{\i (\LambdaAst \j + \talpha) \cdot \tx} \d{\talpha}
  = \frac{1}{2\pi} \int_{\R^2} \dhat{\phi}(\tilde{\bm{\mu}}) e^{\i \tilde{\bm{\mu}} \cdot \tx} \d{\tilde{\bm{\mu}}}  
  = \phi(\tx). 
\end{align*}
A straightforward computation shows the same equality for $\J_{\R^2} \circ \J_{\R^2}^{-1}$.  

(b)
Since $\C^\infty_0(\R^2)$ is dense in $\HH^s_r(\R^2)$, it is again sufficient to work with a smooth function $\phi$ of compact support. 
We assume first that $r=m \in\N_0$ is an integer, and exploit~\eqref{eq:equalNorm2} and~\eqref{eq:blochHoppla} to compute that
\begin{align}
  \|  \J_{\R^2} \phi \|_{\HH^m_\0(\Wast; \HH^s_\talpha(\W))}^2 
  & = \sum_{|\bm\gamma| \leq m} \sum_{\j\in\Z^2} (1+|\j|^2)^s \int_\Wast \big| \partial^{\bm\gamma}_\talpha \dhat{\phi}\left( \LambdaAst \j + \talpha \right) \big|^2 \d{\talpha} \nonumber \\
  & \leq
    C_+(\Lambda) \sum_{|\bm\gamma| \leq m}  \sum_{\j\in\Z^2} \int_\Wast (1+|\LambdaAst \j + \talpha|^2)^s \big|  \partial^{\bm\gamma}_\talpha \dhat{\phi}\left( \LambdaAst \j + \talpha \right) \big|^2 \d{\talpha} \nonumber \\
  & = C_+(\Lambda) \sum_{|\bm\gamma| \leq m} \int_{\R^2} (1+|\z|^2)^s \big| \partial^{\bm\gamma}_{\z} \dhat{\phi}\left( \z \right) \big|^2 \d{\z}, \label{eq:aux406}
\end{align}
where $C_+(\Lambda) = \sup_{\talpha \in \Wast, \, \j\in\Z^2} [ (1+|\j|^2)/ (1+|\LambdaAst \j + \talpha|^2)]^s <\infty$.
Using the Leibniz formula 
\begin{equation}\label{eq:leibniz}
  \partial^{\bm\gamma} \left[ (1+|\z|^2)^{s/2} \dhat{\phi}\left( \z \right) \right]
  = \sum_{\N_0^3 \ni \bm\eta \leq \bm\gamma} {\bm\gamma\choose\bm\eta} \partial^{\bm\eta} (1+|\z|^2)^{s/2} \, \partial^{\bm\gamma-\bm\eta} \dhat{\phi}\left( \z \right)
\end{equation}
there holds for all $\bm\gamma \in \N_0^2$ with $|\bm\gamma| \leq m$ that the two squared norms 
\begin{equation} \label{eq:aux410}
  \int_{\R^2} (1+|\z|^2)^s \big| \partial^{\bm\gamma} \dhat{\phi}\left( \z \right)\big|^2 \d{\z}
  \simeq \int_{\R^2} \bigg| \partial^{\bm\gamma} \left[ (1+|\z|^2)^{s/2} \dhat{\phi}\left( \z \right) \right] \bigg|^2 \d{\z}
\end{equation}
are equivalent. 
We conclude that $\| \J_{\R^2} \phi \|_{\HH^s_\0(\W; \HH^m_\talpha(\Wast))} \leq C(\Lambda,m,s) \| \phi \|_{H^s_m(\R^2)}$. 
The reciprocal inequality follows from first replacing the inequality in~\eqref{eq:aux406} by an estimate from below, exchanging $C_+(\Lambda)$ with $C_-(\Lambda) = \inf_{\talpha \in \Wast, \, \j\in\Z^2} [ (1+|\j|^2)/(1+| \LambdaAst \j + \talpha |^2)]^s > 0$, and afterwards exploiting again~\eqref{eq:aux410}. 
Thus, $\J_{\R^2}$ is an isomorphism from $\HH^s_m(\R^2)$ into $\HH^m_\0(\Wast; \HH^s_\talpha(\W))$. 

For non-integer $r \in (0,\infty) \setminus \N_0$, the corresponding property for $\HH^s_r(\R^2)$ follows from interpolation theory applied to the Sobolev spaces $\HH^s_r(\R^2)$ and $\HH^r_\0(\Wast; \HH^s_\talpha(\W))$, see, e.g.,~\cite{McLea2000, Lions1972}. 
For negative $r<0$, we exploit first that $\HH^s_r(\R^2)$ and $\HH^s_{-r}(\R^2)$ are dual to each other with respect to the scalar product 
\[
  (\psi,\phi) \mapsto \int_{\R^2} (1+|\z|^2)^s \dhat{\psi}(\z) \overline{\dhat{\phi}(\z)} \d{\z}
  \quad \text{for $\psi,\phi \in \HH^s(\R^2)$ with $s\in\R$,} 
\]
and second the duality of $\HH^r_\0(\Wast; \HH^s_\talpha(\W))$ and $\HH^{-r}_\0(\Wast; \HH^s_\talpha(\W))$ with respect to the scalar product of $L^2(\Wast; \HH^s_\talpha(\W))$. 
Thus, a duality argument implies that $\J_{\R^2}$ is an isomorphism from $\HH^s_r(\R^2)$ into $\HH^r_\0(\Wast; \HH^s_\talpha(\W))$ for all $r\in \R$.  

Finally, the inversion formula for the Bloch transform extends from $C^\infty_0(\R^2)$ to all Sobolev spaces $\HH^s_r(\R^2)$ with $s,r \in \R$ by the density of $C^\infty_0(\R^2)$ in each of these spaces.
\end{proof}

The last proof shows in particular that the Bloch transform of $\phi \in \HH^s_r(\R^2)$ can equivalently be represented by the Fourier transform $\dhat{\phi}$ of $\phi$ as in~\eqref{eq:blochHoppla}.
Further, the adjoint of $\J_{\R^2}$ with respect to the scalar product of $L^2(\Wast; L^2(\W))$ equals its inverse:
For $\phi \in \HH^s_r(\R^2)$ and $\psi \in \HH^{-r}_\0(\Wast; \HH^{-s}_\talpha(\W))$, there holds by $\talpha$-quasiperiodicity and density of smooth functions with compact support in $\HH^s_r(\R^2)$ that 
\begin{align*} 
  \langle \J_{\R^2} \phi, \psi \rangle_{L^2(\Wast; L^2(\W))}
  & = \frac{|\det \Lambda|}{2\pi}^{1/2} \int_\Wast \int_\W \sum_{\j\in \Z^2} \phi(\tx + \Lambda \j) e^{-\i \talpha \cdot \Lambda \j}  \, \overline{\psi(\talpha, \tx)} \d{\tx} \d{\talpha} \\
  & = \frac{|\det \Lambda|}{2\pi}^{1/2} \sum_{\j\in \Z^2} \int_\W \phi(\tx + \Lambda \j) \int_\Wast \overline{\psi(\talpha, \tx) e^{\i\talpha \cdot \Lambda \j}}  \d{\talpha} \d{\tx} \\
  & = \frac{|\det \Lambda|}{2\pi}^{1/2} \int_{\R^2} \phi(\tilde{\y}) \int_\Wast \overline{\psi(\talpha,\tilde{\y}) e^{\i\talpha \cdot \Lambda \j}}  \d{\talpha} \d{\tilde{\y}} 
  = \langle \phi, \J_{\R^2}^{-1}\psi \rangle_{L^2(\R)}.
\end{align*}

\begin{corollary}
  \label{th:propertiesBlochR}
For $s,r \in \R$, the adjoint $\J_{\R^2}^\ast$ of the Bloch transform $\J_{\R^2}: \, \HH^s_r(\R^2) \to \HH^r_\0(\Wast; \HH^s_\talpha(\W))$ with respect to the scalar product of $L^2(\Wast; L^2(\W))$ equals $\J_{\R^2}^{-1}$.
%
\end{corollary}

\section{The Bloch Transform on Periodic Surfaces}\label{se:blochSurf}
We define now an analogous partial Bloch transform on Sobolev spaces defined on periodic surfaces. 
To this end, we assume that ${\bm\zeta}: \, \R^3 \to \R^3$ is a $\Lambda$-periodic Lipschitz diffeomorphism, i.e., ${\bm\zeta}$ and its inverse ${\bm\zeta}^{-1}$ are Lipschitz continuous and $\tilde{\y} \mapsto {\bm\zeta}(\tilde{\y},\y_3)$ is $\Lambda$-periodic for all $\y_3\in \R$. 
The inverse mapping ${\bm\zeta}^{-1}$ to ${\bm\zeta}$ is hence Lipschitz continuous, $\Lambda$-periodic, and onto as well.  
This Lipschitz diffeomorphism defines the Lipschitz surface 
\[
  \Gamma := \big\{ {\bm\zeta}(\tilde{\y},0): \, \tilde{\y} \in \R^2 \big\}. 
\]
Suitably shifting ${\bm\zeta}$, we can without loss of generality assume that $\tx \mapsto {\bm\zeta}(\tx , 0)$ is bijective from the fundamental domain $\W \subset \R^2$ onto $\Gamma_\Lambda := \big\{ \x\in \Gamma: \, \tx \in \W \big\}$.
By shifting ${\bm\zeta}$ analogously in $\y_3$, we can further assume that there is $H_0>0$ such that 
\begin{equation}\label{eq:H0}
  0 < \mathrm{ess }\inf_{\tilde{\y} \in \R^2} {\bm\zeta}(\tilde{\y},0) \cdot \e^{(3)}  < H_0  
  \quad \text{for all $\tilde{\y} \in \R^2$, such that $\Gamma \subset \R^2 \times (0,H_0)$,} 
\end{equation}
and that ${\bm\zeta}(\x) =  \x$ for $\x_3 \geq H_0$. 
Following~\cite[Chapter 3]{McLea2000} we further introduce Sobolev spaces on the $\Lambda$-periodic surface $\Gamma$. 
For $\phi: \, \Gamma \to\C$ we define $\phi_{\bm\zeta}: \, \R^2 \to \C$ by 
\[
  \phi_{\bm\zeta}(\tilde{\y}) := \phi({\bm\zeta}(\tilde{\y},0)) 
  \quad \text{ for $\tilde{\y} \in \R^2$, i.e., } 
  \quad 
  \phi(\x) = \phi_{\bm\zeta}( \tilde{\bm\zeta}^{-1}(\x)) \quad \text{  for $x\in \Gamma$.} 
\]
(For simplicity, we write $\tilde{\bm\zeta}^{-1} = (\tilde{\bm\zeta}^{-1}_1,\tilde{\bm\zeta}^{-1}_2)^\top$ for the first two components of ${\bm\zeta}^{-1}$.)
Then, 
\begin{equation}
  \label{eq:liftHs}
  \HH^s_r(\Gamma) := \big\{ \phi: \, \Gamma \to \C \text{ such that }  
    \phi_{\bm\zeta} \in \HH^s_r(\R^2) \big\}, \quad 0 \leq s \leq 1, \, r \in \R, 
\end{equation}
with norm $\| \phi \|_{\HH^s_r(\Gamma)} := \| \phi_{\bm\zeta} \|_{\HH^s_r(\R^2)}$.   
For negative $r<0$ and $-1 \leq s<0$, the spaces $\HH^s_r(\Gamma)$ are defined by duality with respect to the scalar product 
\[
  (\phi, \psi)_{\Gamma} := \int_{\Gamma} \phi \, \ol{\psi} \dS  
  = \int_{\R^2} \phi_{\bm\zeta}(\tilde{\y}) \ol{\psi}_{{\bm\zeta}}(\tilde{\y}) \, | \partial_{\y_1} {\bm\zeta}(\tilde{\y}) \times \partial_{\y_2} {\bm\zeta}(\tilde{\y},0)| \d{\tilde{\y}} 
\]
of $L^2(\Gamma)$. 
The range of $s \in [-1,1]$ is limited as the surface is merely assumed to be Lipschitz smooth, i.e., ${\bm\zeta}$ possesses in general merely first-order essentially bounded weak derivatives.
(If ${\bm\zeta}$ belongs to $C^{m-1,1}(\R^2, \R^3)$ for $m\in\N$, it makes sense to analogously define $\HH^s_r(\Gamma)$ for $s\in [-m, m]$, see~\cite[Chapter 3]{McLea2000}.) 

We further introduce periodic Sobolev spaces $\HH^s_\talpha(\Gamma_\Lambda)$, for $-1 \leq s \leq 1$ by lifting $\HH^s_\talpha(\W)$ to $\Gamma_\Lambda$, 
\begin{equation}
  \label{eq:liftHsPeriodic}
  \HH^s_\talpha(\Gamma_\Lambda) := 
  \big\{ \phi : \, \Gamma \to \C \text{ such that } 
          \tilde{\y} \mapsto \phi_{\bm\zeta}(\tilde{\y}) \in \HH^s_\talpha(\W) \big\}. 
\end{equation}
Norm and scalar product in $\HH^s_\talpha(\Gamma_\Lambda)$ are defined via the same quantities in $\HH^s_\talpha(\W)$, e.g., $\| \phi \|_{\HH^s_\talpha(\Gamma_\Lambda)} := \| \phi_{\bm\zeta} \|_{\HH^s_\talpha(\W)}$. 
For $\Lambda$-periodic functions $\psi$ mapping $\R^2$ into $\HH^s_\talpha(\Gamma_\Lambda)$, we finally define $\HH^r_\0(\Wast; \HH^s_\talpha(\Gamma_\Lambda))$ as in the case of $\Gamma^0 \cong \R^2$, see~\eqref{eq:HrHs}. 

The Bloch transform $\J_{\Gamma}$ of $\phi \in \HH^s_r(\Gamma)$ with $s \in [-1,1]$ and $r\in \R$ then equals
\begin{equation}
  \label{eq:BlochGamma}
  \J_{\Gamma} \phi \big(\talpha, \x) 
  := \frac{|\det\Lambda|}{2\pi}^{1/2} \sum_{\j\in \Z^2} \phi \left( \begin{smallmatrix} \tx +\Lambda \j \\ \x_3 \end{smallmatrix} \right) e^{-\i \talpha \cdot \Lambda \j}, 
  \quad  \x = \left( \begin{smallmatrix} \tx \\ \x_3 \end{smallmatrix} \right) \in \Gamma, \, \talpha \in \R^2.
\end{equation}
As ${\bm\zeta}|_{\{ \y_3 = 0\}}$ is bijective from $\{ \y_3 = 0\}$ onto $\Gamma$ and as ${\bm\zeta}^{-1}|_{\Gamma}$ is the inverse of this bijection, there is for each $x\in\Gamma$ a unique $\tilde{\y} = \tilde{\bm\zeta}^{-1}(\x) \in \R^2$ such that ${\bm\zeta}(\tilde{\y},0) = \x$. 
$\Lambda$-periodicity of ${\bm\zeta}$ further implies that 
\[
  \left( \begin{smallmatrix} \tx +\Lambda \j \\ \x_3 \end{smallmatrix} \right)
  = {\bm\zeta} (\tilde{\y} + \Lambda,0)
  = {\bm\zeta} \left( {\bm\zeta}^{-1}\left( \begin{smallmatrix} \tx +\Lambda \j \\ \x_3 \end{smallmatrix}  \right) \right) 
  \quad \text{for all } \x\in\Gamma \text{ and } \j \in \Z^2,
\]
and we conclude that $\J_{\Gamma} \phi \big(\talpha, \x) = \J_{\R^2} \phi_{\bm\zeta} \big(\talpha, \tilde\rho(\x)\big)$. 
Thus, the mapping properties of $\J_{\R^2}$ directly carry over to $\J_{\Gamma}$, as $\HH^s_r(\Gamma)$ and $\HH^s_\talpha(\Gamma_\Lambda)$ are defined by lifting $\HH^s_r(\R^2)$ and $\HH^s_\talpha(\W)$ to $\Gamma$, respectively, see~\eqref{eq:liftHs} and~\eqref{eq:liftHsPeriodic}.
In particular, $\J_{\Gamma} \phi$ is as well the lifting of $\J_{\R^2} \phi_{\bm\zeta}$ to $\Gamma$.  
Thus, the next result is a corollary of Theorem~\ref{th:BlochR}. 

\begin{theorem}
  \label{th:Bloch}
  For $s \in [-1,1]$ and $r \in \R$, the Bloch transform $\J_{\Gamma}$ is an isomorphism between $\HH^s_r(\Gamma)$ and $\HH^r_\0(\Wast; \HH^s_\talpha(\Gamma_\Lambda))$ that is an isometry for $s=r=0$. 
  Its inverse equals its $L^2$-adjoint and is given by 
  \[
    \J_{\Gamma}^{-1} \psi \left(\x + \left( \begin{smallmatrix} \Lambda \j \\ 0 \end{smallmatrix} \right) \right)
    = \frac{|\det\Lambda|}{2\pi}^{1/2} 
    \int_\Wast \psi(\talpha, \x) e^{\i \talpha \cdot \Lambda \j}\d{\talpha}
    \quad \text{for }  \x \in \Gamma_\Lambda, \, \j\in\Z^2 .
  \]
\end{theorem}
\begin{proof}
We merely show that $\J_{\Gamma}$ is an isometry between $\HH^0_\0(\Gamma) = L^2(\Gamma)$ and $\HH^0_\0(\Wast; \HH^0_\talpha(\Gamma_\Lambda)) = L^2(\Wast; L^2(\Gamma_\Lambda))$, which follows from the isometry property of $\J_{\R^2}$: 
For $\phi \in L^2(\Gamma)$, we exploit the definition of $\J_{\Gamma} \phi$ and that $\| \phi \|_{\HH^0_\0(\Gamma)} = \| \varphi_{\bm\zeta} \|_{L^2(\R^2)}$, to obtain that 
$\| \J_{\Gamma} \phi \|_{L^2(\Wast; L^2(\Gamma_\Lambda))}
  =  \| (\J_{\R^2} \phi_{\bm\zeta})(\cdot, \tilde{\bm\zeta}^{-1}(\cdot)) \|_{L^2(\Wast; L^2(\Gamma_\Lambda))} 
  =  \| \phi_{\bm\zeta} \|_{L^2(\R^2)}
  =  \| \phi \|_{L^2(\Gamma)}$.
\end{proof}

\section{The Bloch Transform on Periodic Strips}\label{se:blochDom}
To define a Bloch transform on periodic domains, let us use the diffeomorphism $\bm\zeta$ from the last section to define two domains above the periodic surface $\Gamma$ by
\[  
\Omega := \big\{ {\bm\zeta}(\y): \, \y\in\R^3, \, \y_3>0  \big\} 
  \quad \text{and} \quad 
  \Omega_H := \big\{ {\bm\zeta}(\y): \, \y\in\R^3, \, 0<\y_3, \, {\bm\zeta}(\y)\cdot \e^{(3)} < H \big\}    
\]
for $H \geq H_0$ (see~\eqref{eq:H0}.
Both domains are hence unbounded Lipschitz domains and $\Lambda$-periodic in $\tx$. 
Since $H \geq H_0$, the boundary of the truncated domain $\Omega_H$ moreover is the union of $\Gamma$ and $\Gamma_H = \{ \x_3 = H \}$. 
We also introduce, roughly speaking, the restriction of $\Omega$ and $\Omega_H$ to the fundamental domain of periodicity $\W$, 
\[
  \Omega_\Lambda := \big\{ \x\in\Omega: \, \tx \in \W \big\} 
  \quad \text{and} \quad 
  \Omega_H^\Lambda := \big\{ \x\in\Omega_H: \, \tx \in \W \big\}. 
\]
For smooth functions $u: \, \Omega_H \to \C^d$ with compact support we define the horizontal Bloch transform $\J_{\Omega}$ by 
\begin{equation}
  \label{eq:BlochOmega}
  \J_{\Omega} u(\talpha, \x) 
  := \frac{|\det\Lambda|}{2\pi}^{1/2} \sum_{\j\in \Z^2} u(\tx + \Lambda \j,\x_3) e^{-\i \talpha \cdot \Lambda \j},
  \quad  \x \in \Omega_H, \, \talpha \in \R^2, \, H \geq H_0.
\end{equation}
Obviously, restricting the partial transform $\J_{\Omega} u(\talpha,\cdot)$ to $\Gamma$ yields the Bloch transform $\J_{\Gamma}$ applied to the restriction $u|_{\Gamma}$ and $\J_{\Omega} u(\talpha, \x) = \J_{\R^2}\big( u(\cdot, \x_3) \big)(\talpha,\tx)$ holds for for all $\x \in \Omega$.
Again, mapping properties of $\J_\Omega$ rely on suitable function spaces. 
For $m\in\N_0$, 
\begin{equation}
  \HH^m(\Omega_H) := \left\{ u \in L^2(\Omega_H): \, \partial^{\bm\gamma} u \in L^2(\Omega_H) \text{ for } \bm\gamma \in \N_0^3, \, |\bm\gamma|\leq m \right\} 
\end{equation}
and note that the weighted analogues of these Sobolev spaces for polynomial weights are   
\begin{equation} \label{eq:HMR}
  \HH^m_r(\Omega_H) := \left\{ u \in L^2(\Omega_H): \, (1+|\tx|^2)^{r/2} u(\x) \in \HH^m(\Omega_H) \right\}\quad r\in\R, m\in\N_0.
\end{equation}
These weighted spaces are normed by $\| u \|_{\HH^m_r(\Omega_H)} =  \|  \x \mapsto (1+|\tx|^2)^{r/2} \, u(\x) \big\|_{\HH^m(\Omega_H)}^{1/2}$, where
$\| u \|_{\HH^m(\Omega_H)}^2 = \sum_{\bm\gamma \in \N_0^3, \, |\bm\gamma|\leq m} \| \partial^{\bm\gamma} u \|^2_{L^2(\Omega_H)}$ is, for $m\in\N_0$, the squared norm of $\HH^m(\Omega_H)$.   
The Leibniz formula, see~\eqref{eq:leibniz}, implies for $s=m\in\N_0$ that 
\begin{equation} \label{eq:equivNormHOmega}
 u\mapsto 
 \sum_{\bm\gamma \in \N_0^3, \, |\bm\gamma|\leq m} \|  \x \mapsto (1+|\tx|^2)^{r/2}  [\partial^{\bm\gamma} u(\x)] \|^2_{L^2(\Omega_H)}
 = \sum_{\bm\gamma \in \N_0^3, \, |\bm\gamma|\leq m} \| \partial^{\bm\gamma} u \|^2_{L^2_r(\Omega_H)} 
\end{equation}
defines an equivalent (squared) norm on $\HH^m_r(\Omega_H)$. 
Thus, $u\in \HH^m_r(\Omega_H)$ possesses weak derivatives in $L^2_r(\Omega_H)$ up to order $m\in\N$. 

For arbitrary $s>0$, $\HH^s_r(\Omega_H)$ is defined by interpolation, see~\cite{McLea2000}. 
If we denote the closure of $C^\infty_0(\Omega_H)$ in the norm of $\HH^s_r(\Omega_H)$ by $\tilde\HH^s_r(\Omega_H)$, then the dual space of $\HH^s_r(\Omega_H)$ for $s \geq 0$ with respect to the scalar product 
\[
  (u,v) \mapsto \int_{\Omega_H} (1+|\tx|^2)^r u \overline{v} \d{\x}
  \quad \text{in } L^2_r(\Omega_H) 
\]
is denoted by $\tilde\HH^{-s}_r(\Omega_H)$. 
The corresponding dual of $\tilde\HH^s_r(\Omega_H)$ is denoted by $\HH^{-s}_r(\Omega_H)$.
(Note that $\HH^0_0(\Omega_H) = L^2(\Omega_H)$.)

To state mapping properties of $\J_{\Omega}$, we introduce for $m\in\N_0$ the periodic Sobolev spaces $\HH^m_\talpha(\Omega_H^\Lambda)$ of functions $u$ in $\Omega_H$ that are $\talpha$-quasiperiodic in $\tx$ with respect to $\Lambda$, such that  
\begin{equation} \label{eq:HmAlpha}
  \| u \|_{\HH^m_\talpha(\Omega_H^\Lambda)} 
  := \Bigg[ \sum_{\bm\gamma \in \N^3_0, |\bm\gamma|\leq m} \| \partial^{\bm\gamma} u \|^2_{L^2(\Omega_H^\Lambda)} \bigg]^{1/2} < \infty 
\end{equation}
is finite. 
For arbitrary $s>0$, the intermediate spaces $\HH^s_\talpha(\Omega_H^\Lambda)$ are defined by interpolation and the closure of smooth $\talpha$-quasiperiodic functions in $\Omega_H$ with compact support in $\Omega_H$ in the norm of $\HH^s_\talpha(\Omega_H^\Lambda)$ is $\tilde\HH^s_\talpha(\Omega_H^\Lambda)$. 
The dual spaces of $\HH^s_\talpha(\Omega_H^\Lambda)$ and $\tilde\HH^s_\talpha(\Omega_H^\Lambda)$ for the scalar product of $L^2(\Omega_H^\Lambda) =\HH^0_\talpha(\Omega_H^\Lambda)$ are $\tilde\HH^{-s}_\talpha(\Omega_H^\Lambda)$ and $\HH^{-s}_\talpha(\Omega_H^\Lambda)$, respectively. 
Finally, the product spaces $\HH^r_\0(\Wast; \HH^s_\talpha(\Omega_H^\Lambda))$ and $\HH^r_\0(\Wast; \tilde\HH^s_\talpha(\Omega_H^\Lambda))$ are defined as explained in Remark~\ref{rem:X}(b).


\begin{lemma}\label{th:BlochPartDeri}
  If $u \in \HH^m_r(\Omega_H)$ for $m\in\N$, then $\J_{\Omega} u$ possesses weak partial derivatives with respect to $\x \in \Omega_H^\Lambda$ in $L^2_r(\Omega_H^\Lambda)$ up to order $m\in\N$. 
  If $\bm\gamma \in \N_0^3$ with $|\bm\gamma| \leq m$, then 
  \begin{equation}\label{eq:transDeri}
    \partial_\x^{\bm\gamma} \J_{\Omega} u(\talpha, \x)
    =\J_{\Omega} \left[\partial^{\bm\gamma} u \right](\talpha, \x)
    \quad \text{in } L^2_r(\Omega_H^\Lambda). 
  \end{equation}
  If $\x \mapsto \tx^{\tilde{\bm\gamma}} u(\x) \in L^2(\Omega_H^\Lambda)$, then 
  \begin{equation}\label{eq:aux639} 
    \sum_{\tilde{\bm\mu} \in\N_0^2, \, \tilde{\bm\mu} \leq \tilde{\bm\gamma}} 
    {\tilde{\bm\gamma} \choose \tilde{\bm\mu}} \,  \big( -\i \tx \big)^{\tilde{\bm\gamma} - \tilde{\bm\mu}} \, \partial^{\tilde{\bm\mu}}_\talpha \J_{\Omega} u(\talpha, \x) 
   = \J_{\Omega} \big[ \y \mapsto (-\i\tilde\y)^{\tilde{\bm\gamma}}  u \big](\talpha, \x)
    \quad \text{in } L^2(\Omega_H^\Lambda).     
  \end{equation}
(As usual, $\tilde{\bm\beta}^{\tilde{\bm\gamma}} := \bm\beta_1^{\bm\gamma_1} \, \bm\beta_2^{\bm\gamma_2}$ for $\tilde{\bm\beta} \in \C^2$ and $\tilde{\bm\gamma} \in \N_0^2$.)
\end{lemma}

\begin{proof}
Consider the set of all smooth functions in $C^\infty(\overline{\Omega_H})$ with compact support included in $\overline{\Omega_H}$, such that additionally all partial derivatives continuously extend to $\overline{\Omega_H}$. 
This set forms a dense subset of $\HH^m_r(\Omega_H)$.
Thus, it is sufficient to prove the claimed identities for such infinitely often differentiable functions: 
The identity stated in~\eqref{eq:transDeri} 
holds by exchanging derivatives and the series in the definition of $\J_{\Omega}$, which possesses only a finite number of non-zero terms by the compactness of the support of the argument of $\J_{\Omega}$. 
Moreover,~\eqref{eq:transDeri} implies the statement on the differentiability of $\J_{\Omega}$ with respect to $\x$. 
Concerning~\eqref{eq:aux639}, we compute for $\tilde{\bm\gamma} \in \N_0^2$ with $|\tilde{\bm\gamma}| \leq m$ that 
\begin{equation} \label{eq:bloDerCom}
\begin{split}
  \partial^{\tilde{\bm\gamma}}_\talpha \big[ e^{-\i \talpha \cdot \tilde{x}} \J_{\Omega} u(\talpha, \x) \big]
  &= \frac{|\det\Lambda|}{2\pi}^{1/2} \sum_{\j\in \Z^2} u(\tx + \Lambda \j,\x_3) \partial^{\tilde{\bm\gamma}}_\talpha e^{-\i \talpha \cdot (\tx+\Lambda \j)}\\
  & =  \frac{|\det\Lambda|}{2\pi}^{1/2} \sum_{\j\in \Z^2} u(\tx + \Lambda \j,\x_3) \big(-\i(\tx+\Lambda \j)\big)^{\tilde{\bm\gamma}} e^{-\i \talpha \cdot (\tx+\Lambda \j)} \\
  & = e^{-\i \talpha \cdot \tilde{x}} \J_{\Omega} \big[ \y \mapsto (-\i\tilde\y)^{\tilde{\bm\gamma}}  u(\y) \big](\talpha, \x), \quad \talpha \in\Wast, \x \in \Omega_H^\Lambda,
\end{split}
\end{equation}
which implies~\eqref{eq:aux639} by applying the Leibniz formula, see~\eqref{eq:leibniz}, to the left-hand side, noting that $(\talpha,\tx) \mapsto \exp(-\i \talpha \cdot \tx)$ is a smooth function in $\Wast$. 
\end{proof}


\begin{theorem}\label{th:BlochOmega}
The Bloch transform $\J_{\Omega}$ extends to an isomorphism between $\HH^s_r(\Omega_H)$ and $\HH^r_\0(\Wast; \HH^s_\talpha(\Omega_H^\Lambda))$ as well as between $\tilde\HH^s_r(\Omega_H)$ and $\HH^r_\0(\Wast; \tilde\HH^s_\talpha(\Omega_H^\Lambda))$ for all $s,r \in \R$. 
  For $s=r=0$, $\J_{\Omega}$ is an isometry between $L^2(\Omega_H) = \HH^0_0(\Omega_H)$ and $L^2(\Wast; L^2(\Omega_H^\Lambda)) = \HH^0_\0(\Wast; \HH^0_\talpha(\Omega_H^\Lambda))$. 
  Its inverse transform equals 
  \begin{equation}
    \label{eq:JOmegaInverse}
    \J_{\Omega}^{-1} w \left(\x+ \left( \begin{smallmatrix} \Lambda \j \\ 0 \end{smallmatrix} \right) \right)
    = \frac{|\det\Lambda|}{2\pi}^{1/2} 
    \int_\Wast w(\talpha, \x) e^{\i \talpha \cdot \Lambda \j}\d{\talpha}
    \quad \text{for }  \x \in \Omega_H^\Lambda, \, \j \in \Z^2 .
  \end{equation}
\end{theorem}
\begin{proof}
To prove that $\J_{\Omega}$ is an isometry from $L^2(\Omega_H)$ into $L^2(\Wast; L^2(\Omega_H^\Lambda))$ it is by density sufficient to consider a smooth function $u \in C^\infty_0(\Omega_H)$ with compact support included in $\Omega_H^\Lambda$. 
We extend $u$ by zero to all of $\R^3$ and exploit the isometry of $\J_{\R^2}$ on $L^2(\R^2)$, 
\begin{align}
  & \|  \J_{\Omega} u \|_{L^2(\Wast; L^2(\Omega_H^\Lambda))}^2 
   = \int_\Wast \int_{\Omega_H^\Lambda} \big|  \J_{\Omega} u \left( \talpha, \x \right) \big|^2 \d{\x} \d{\talpha} \label{eq:JOmIso} \\ 
  & = \int_0^H \int_\Wast  \int_\W \big|  \J_{\R^2} u \left( \talpha, \big( \begin{smallmatrix}\tx \\ \x_3 \end{smallmatrix}\big) \right) \big|^2 \d{\tx} \d{\talpha} \d{x_3} 
   = \int_0^H \int_{\R^2} \big| u \big( \begin{smallmatrix}\tx \\ \x_3 \end{smallmatrix}\big) \big|^2 \d{\tx} \d{x_3}
  = \| u \|_{L^2(\Omega_H)}^2 \nonumber
\end{align}
as $u$ vanishes outside the domain $\Omega_H$. 

Next, we show that $\J_{\R^2}$ extends to an isomorphism from $\HH^s_r(\Omega_H)$ onto $\HH^r_\0(\Wast; \HH^s_\talpha(\Omega_H^\Lambda))$. 
We first treat the case where both indices $s$ and $r$ are entire numbers and choose $\ell=s \in\N_0$ and $m=r\in\N_0$. 
We again exploit that smooth functions in $C^\infty(\overline{\Omega_H})$ such that their support is compact and included in $\overline{\Omega_H}$, and such that all their partial derivatives extend continuously to $\overline{\Omega_H}$, are dense in $\HH^\ell_m(\Omega_H)$.  
For such a smooth function $u$, we find that 
\begin{align*}
  \| \J_{\Omega} u \|_{\HH^\ell_\0(\Wast; \HH^m_\talpha(\Omega_H^\Lambda))}^2 
   & = \sum_{\tilde{\bm\gamma} \in \N^2_0,\, |\tilde{\bm\gamma}| \leq \ell} \int_\Wast  \sum_{\bm\eta \in \N^3_0, \, |\bm\eta| \leq m} \int_{\Omega_H^\Lambda} \big| \partial^{\tilde{\bm\gamma}}_\talpha \partial_\x^{\bm\eta} \J_{\Omega} u \left( \talpha, \x \right) \big|^2 \d{\x} \d{\talpha}  \\
  &   \stackrel{\eqref{eq:transDeri}}{=} 
  \sum_{|\tilde{\bm\gamma}| \leq \ell} \int_\Wast \sum_{|\bm\eta| \leq m} \int_{\Omega_H^\Lambda} \Big| \partial^{\tilde{\bm\gamma}}_\talpha  \J_{\Omega} [\partial^{\bm\eta} u] \left( \talpha, \x \right) \Big|^2 \d{\x} \d{\talpha}  
  =: (\ast).
\end{align*}
By~\eqref{eq:aux639} and the triangle inequality, we note that 
\begin{align*}
  \bigg| \partial^{\tilde{\bm\gamma}}_\talpha  \J_{\Omega} v (\talpha, \x) \bigg|^2
  & = \bigg| \J_{\Omega} [\y \mapsto  (-\i\tilde{\y})^{\bm\gamma} v(\y)]  \left( \talpha, \x \right) 
  - \sum_{\tilde{\bm\mu} \in \N_0^2, \, \tilde{\bm\mu} < \tilde{\bm\gamma}} {\tilde{\bm\gamma} \choose \tilde{\bm\mu}}  \big( -\i \tx \big)^{\tilde{\bm\gamma} - \tilde{\bm\mu}} \, \partial^{\tilde{\bm\mu}}_\talpha \J_{\Omega} v(\talpha, \x)\bigg|^2 \\
  & \leq C(\Lambda, \ell) \left[ \left| \J_{\Omega} [\y \mapsto  (-\i\tilde{\y})^{\tilde{\bm\gamma}} v(\y)]  \left( \talpha, \x \right) \right|^2
  + \sum_{\tilde{\bm\mu} < \tilde{\bm\gamma}} \Big| \partial^{\tilde{\bm\mu}}_\talpha \J_{\Omega} v(\talpha, \x)\Big|^2 \right] \\
  & \leq C(\Lambda, \ell) \sum_{\tilde{\bm\mu} \leq \tilde{\bm\gamma}} \Big| \J_{\Omega} [\y \mapsto  (-\i\tilde{\y})^{\tilde{\bm\mu}} v(\y)]  \left( \talpha, \x \right) \Big|^2, \quad \talpha \in\Wast, \x \in \Omega_H^\Lambda,
\end{align*}
with $C(\Lambda,\ell) \leq (1+|\LambdaAst|)^{\ell} (\ell!)^2$. 
As $\J_{\Omega}$ is an isometry from $L^2(\Omega_H)$ into $L^2(\Wast; L^2(\W))$ by~\eqref{eq:JOmIso}, we conclude that 
\begin{align}
  & (\ast) 
  \leq C(\Lambda, \ell)  \sum_{|\tilde{\bm\gamma}| \leq \ell} \sum_{|\bm\eta| \leq m} \sum_{\tilde{\bm\mu} \leq \tilde{\bm\gamma}}  \int_\Wast  \int_{\Omega_H^\Lambda} \Big| \J_{\Omega} [\y \mapsto  (-\i\tilde{\y})^{\tilde{\bm\mu}} \partial^{\bm\eta} u(\y)]  \left( \talpha, \x \right) \Big|^2 \d{\x} \d{\talpha} \nonumber \\
  & = C(\Lambda, \ell)  \sum_{|\tilde{\bm\gamma}| \leq \ell} \sum_{|\bm\eta| \leq m} \sum_{\tilde{\bm\mu} \leq \tilde{\bm\gamma}}  \int_{\Omega_H} \Big| \tilde{\y}^{\tilde{\bm\mu}} \partial^{\bm\eta} u(\y) \Big|^2 \d{\y} \label{eq:aux805}\\
  & \leq C(\Lambda, \ell) \sum_{|\bm\eta| \leq m}  \int_{\Omega_H} (1+ |\tilde\y|^2)^\ell \big|  \partial^{\bm\eta} u(\y) \big|^2 \d{\y} 
  \leq C(\Lambda,\ell) \| u \|_{\WW^{m}_\ell(\Omega_H)}^2. \nonumber
\end{align}
Equivalence of the norms $\| \cdot \|_{\HH^m_\ell(\Omega_H)}$ and $\| \cdot \|_{\WW^{m}_\ell(\Omega_H)}$ on $\HH^m_\ell(\Omega_H)$, see~\eqref{eq:equivNormHOmega}, now implies that $\| \J_{\Omega} u \|_{\HH^\ell_\0(\Wast; \HH^m_\talpha(\Omega_H^\Lambda))} \leq C(\Lambda,\ell) \| u \|_{\HH^{m}_\ell(\Omega_H)}$, such that $\J_{\Omega}$ is bounded from $\HH^m_\ell(\Omega_H)$ into $\HH^\ell_\0(\Wast; \HH^m_\talpha(\Omega_H^\Lambda))$. 
The reverse inequality is also due to the latter norm equivalence, as for all $u\in C^\infty(\overline{\Omega_H})$ there holds  
\begin{align*}
  \| u \|_{\HH^m_\ell(\Omega_H)}^2 
  &\leq C \sum_{\bm\eta \in \N_0^3, |\bm\eta|\leq m} \int_{\Omega_H} (1+|\tx|^2)^{\ell}  |\partial^{\bm\eta} u(\x)|^2 \d{\x} \\
  & \leq C(\ell) \sum_{|\bm\eta|\leq m} \sum_{\tilde{\bm\gamma} \in \N_0^2, |\tilde{\bm\gamma}|\leq \ell} \int_{\Omega_H} \big| (-\i \tx)^{\tilde{\bm\gamma}} \partial^{\bm\eta} u(\x)\big|^2 \d{\x} \\
  & = C(\ell) \sum_{|\bm\eta|\leq m} \sum_{|\tilde{\bm\gamma}|\leq \ell} \int_{\Wast} \int_{\Omega_H^\Lambda} \big| \J_{\Omega} [ y \mapsto (-\i \tilde\y)^{\tilde{\bm\gamma}} \partial^{\bm\eta} u(\y)](\talpha; \x) \big|^2 \d{\x} \d{\talpha} = (\bullet). 
\end{align*}
Next, the identity~\eqref{eq:aux639} implies that 
\begin{align*}
  (\bullet)
  & \leq C(\ell) \sum_{|\bm\eta|\leq m} \sum_{|\tilde{\bm\gamma}|\leq \ell} \int_{\Wast} \int_{\Omega_H^\Lambda} \sum_{\tilde{\bm\mu} \leq \tilde{\bm\gamma}}  \Big| {\tilde{\bm\gamma} \choose \tilde{\bm\mu}} \, \big( -\i \tx \big)^{\tilde{\bm\gamma} - \tilde{\bm\mu}} \, \partial^{\tilde{\bm\mu}}_\talpha \partial^{\bm\eta}_{\x} \J_{\Omega} u(\talpha, \x) 
 \Big|^2 \d{\x} \d{\talpha}   \\
 & \leq C(\ell) \sum_{|\bm\eta|\leq m} \sum_{|\tilde{\bm\gamma}|\leq \ell} \int_{\Wast} \int_{\Omega_H^\Lambda} \sum_{\tilde{\bm\mu} \leq \tilde{\bm\gamma}}  \Big| \partial^{\tilde{\bm\mu}}_\talpha \partial^{\bm\eta}_{\x} \J_{\Omega} u(\talpha, \x) 
 \Big|^2 \d{\x} \d{\talpha} 
 \leq C(\ell) \| \J_{\Omega} u \|_{\HH^\ell_\0(\Wast; \HH^m_\talpha(\Omega_H^\Lambda))}^2 
\end{align*}
for a constant $C(\ell)$ that grows at most as $(1+|\LambdaAst|^2)^\ell (\ell!)^2$. 
%
%
The latter estimate together with~\eqref{eq:aux805} shows that $\J_{\Omega}$ is an isomorphism from $\HH^m_\ell(\Omega_H)$ into $\HH^\ell_\0(\Wast; \HH^m_\talpha(\Omega_H^\Lambda))$ for all $m\in\N_0$ and $\ell \in\N_0$.
Obviously, the above computations also hold for a smooth function $u \in C^\infty_0(\Omega_H^\Lambda)$, such that $\J_{\Omega}$ is also an isomorphism between $\tilde\HH^m_\ell(\Omega_H)$ into $\HH^\ell_\0(\Wast; \tilde\HH^m_\talpha(\Omega_H^\Lambda))$, defined as closures of smooth and compactly supported functions. 
Interpolation first in $m \in\N_0$ and in $\ell \in \N_0$ extends these two results to all real and positive indices $s \geq 0$ and $r \geq 0$.
 
Next, duality of $\HH^s_r(\Omega_H)$ and $\HH^{s}_{-r}(\Omega_H)$ for the scalar product of $\HH^s(\Omega_H)$, and of $\HH^r_\0(\Wast; \HH^s_\talpha(\Omega_H^\Lambda))$ and $\HH^{-r}_\0(\Wast; \HH^{s}_\talpha(\Omega_H^\Lambda))$ for the scalar product of $L^2(\Wast; \HH^{s}_\talpha(\Omega_H^\Lambda))$ for $r \geq 0$ and $s \geq 0$ yields the claimed isomorphy properties of $\J_{\Omega}$ for the entire range of $r \in\R$. 
Another duality argument in $s$ for the scalar product of $L^2(\Omega_H)$ finally yields the entire range $s,r \in\R$. 
The arguments for $\tilde\HH^{s}_{r}(\Omega_H)$ and $\tilde\HH^s(\Omega_H)$ are fully analogous.
\end{proof}

\section{Periodic Surface Scattering}\label{se:perioSurfScatt}
The Bloch transform reduces acoustic scattering problems from periodic surfaces with non-periodic boundary data to a family of decoupled quasiperiodic scattering problems. 
To illustrate this reduction, we consider wave scattering from the $\Lambda$-periodic surface $\Gamma = \{ {\bm\zeta}(\tilde{\y},0): \, \tilde{\y} \in \R^2 \}$ in the periodic domain of propagation ${\Omega} = \big\{ {\bm\zeta}(\y): \, \y\in\R^3, \, \y_3>0  \big\}$, defined in Sections~\ref{se:blochSurf} and~\ref{se:blochDom}. 

Considering the Helmholtz equation at constant wave number $k>0$ for a scalar function $u$, 
\begin{equation}\label{eq:HE}
  \Delta u + k^2 u = 0 \quad \text{in } \Omega \subset \R^3,  
\end{equation}
we choose a weight parameter $r > -1$ and seek for weak solutions to this problem in 
\[
  \HH^1_{r, \loc}(\Omega) := \left\{ u \in \mathcal{D}'(\Omega): \, \left. u \right|_{\Omega_H} \in \HH^1_r(\Omega_H)  \text{ for all } H \geq H_0 \right\}.
\] 
It is well-known that the corresponding trace space equals $\HH^{1/2}_r(\Gamma)$ and that the trace operator $\left. u \right|_{\partial \Omega}$ is bounded from $\HH^1_r(\Omega_H)$ onto $\HH^{1/2}_r(\Gamma)$. 
For a $\Lambda$-periodic coefficient $\theta_\Lambda \in L^\infty(\Gamma)$, we introduce impedance- or Robin-type boundary conditions with right-hand side $f \in \HH^{-1/2}_r(\Gamma)$, 
\begin{equation}\label{eq:bcHe}
  \frac{\partial u}{\partial \bm{\nu}} - \theta_\Lambda u = f \quad \text{on } \Gamma .
\end{equation}
(Here and in the following, $\bm{\nu}$ is the unit normal to $\Gamma$ that points into $\Omega$.) 
Writing $\dhat{u}(\cdot,H_0)$ for the Fourier transform of the restriction of $u$ to $\Gamma_{H_0}$, see~\eqref{eq:fourierTrafo} and~\eqref{eq:H0}, we require $u$ to satisfy the following angular spectrum representation as a radiation condition, 
\begin{equation}\label{eq:URC}
  u(\x) = \frac{1}{2\pi} \int_{\R^2} e^{\i \tx \cdot \bm\xi + \i \sqrt{k^2 - |\bm\xi|^2} (\x_3-H_0)} \dhat{u}(\bm\xi,H_0) \d{\bm\xi} 
  \quad \text{for } \x_3 > H_0. 
\end{equation}
Here $\sqrt{k^2 - |\bm\xi|^2}  = \i \sqrt{|\bm\xi|^2 - k^2}$ for $|\bm\xi|^2 > k^2$; more generally, we extend the square-root function analytically into the complex plane slit at the negative imaginary axis. 
Note that~\eqref{eq:URC} implies that $u$ satisfies that relation for $H_0$ replaced by any $H>H_0$. 
Restricting the equality in~\eqref{eq:URC} formally to $\Gamma_H$ provides a link between the normal derivative of $u$ on $\Gamma_H$ and the exterior Dirichlet-to-Neumann operator $T^+$, 
\[
  \frac{\partial u}{\partial \x_3}(\tx , H) 
  = \frac{\i}{2\pi} \int_{\R^2} \sqrt{k^2 - |\bm\xi|^2} \, e^{\i \tx \cdot \bm\xi} \, \dhat{u}(\bm\xi,H) \d{\bm\xi}
  =: T^+ \left( u|_{\Gamma_{H}} \right)(\tx , H), 
  \quad 
  H>H_0, 
\]
which is continuous from $\HH^{1/2}_{r'}(\Gamma_H)$ into $\HH^{-1/2}_{r'}(\Gamma_H)$ if $|r'|<1$, see~\cite{Chand2010}.
Thus, the variational formulation of~(\ref{eq:HE}--\ref{eq:bcHe}) for $r>-1$ together with the radiation condition~\eqref{eq:URC} is to find $u \in \HH^1_r(\Omega_H)$ such that 
\begin{equation}\label{eq:varFormHEScal}
  \int_{\Omega_H} \left[ \nabla u \cdot \nabla \overline{v} - k^2 u \,\overline{v} \right] \d{\x}
  + \int_{\Gamma} \theta_\Lambda \, u \, \overline{v} \dS  
  - \int_{\Gamma_H} T^+\left( u|_{\Gamma_{H}} \right) \, \overline{v} \dS 
  = \int_{\Gamma} f \, \overline{v} \dS    
\end{equation}
for all $v \in \HH^1(\Omega_H)$ with compact support in $\overline{\Omega_H^\Lambda}$.
(As $T^+$ is continuous between $\HH^{\pm1/2}_{r'}(\Gamma_H)$ for $|r'|<1$, the latter equation is well-defined for test functions with compact support.) 
The next theorem illustrates that the Bloch transform of $u$ weakly solves a transformed Helmholtz equation with periodic boundary conditions. 
It relies on a periodic Dirichlet-to-Neumann operator $T_\talpha^+$ on $\Gamma_H^\Lambda = \{  \x \in \Gamma_H: \, \tx \in \W \} \subset \Gamma_H$, which is continuous from $\HH^{1/2}_\talpha(\Gamma_H^\Lambda)$ into $\HH^{-1/2}_\talpha(\Gamma_H^\Lambda)$,
\[
   T_\talpha^+ \left( \varphi \right)
   = \i \sum_{\j\in\Z^2} \beta_j(\talpha,k) \, \hat{\varphi}(\j) \, e^{\i (\LambdaAst \j+\talpha) \cdot \tx}
   \quad \text{for } 
   \varphi = \sum_{\j\in\Z^2} \hat{\varphi}(\j) e^{\i (\LambdaAst \j + \talpha) \cdot \tx},
\]
where, as above, $\beta_j(\talpha,k) := \sqrt{k^2 - |\LambdaAst \j + \talpha|^2}$ is $\i \, |k^2 - |\LambdaAst \j + \talpha|^2|^{1/2}$ if $k<|\LambdaAst \j + \talpha|$. 

\begin{theorem}\label{th:equiScalarPerio}
For $|r|<1$, a function $u \in \HH^1_r(\Omega_H^\Lambda)$ solves~\eqref{eq:varFormHEScal} for $f \in \HH^{-1/2}_r(\Gamma)$ if and only if $w := \J_{\Omega} u \in \HH^r_\0(\Wast; \HH^1_\talpha(\Omega_H^\Lambda))$ solves   
\begin{equation}\label{eq:heAlpha}
  \int_\Wast a_{k,\talpha} \big(w(\talpha,\cdot),v(\talpha,\cdot) \big) \d{\talpha}
  = \int_\Wast \int_{\Gamma_\Lambda}  \J_{\Gamma} f(\talpha,\cdot) \, \overline{v(\talpha,\cdot)} \dS \d{\talpha} 
\end{equation} 
for all $v \in \HH^{-r}_\0(\Wast; \HH^1_\talpha(\Omega_H^\Lambda))$, where 
\begin{multline}
  a_{k,\talpha} \big(w_\talpha,v_\talpha\big)
  := \int_{\Omega_H^\Lambda} \Big[ \nabla w_\talpha(\x) \cdot \nabla \overline{v_\talpha(\x)} - k^2 w_\talpha(\x) \, \overline{v_\talpha(\x)} \Big] \d{\x} 
  + \int_{\Gamma_\Lambda} \theta_\Lambda \, w_\talpha(\x) \, \overline{v_\talpha(\x)} \dS  \\ 
  - \int_{\Gamma_H^\Lambda} T_\talpha^+ \left( w_\talpha(\x)|_{\Gamma_{H}} \right) \, \overline{v_\talpha(\x)} \dS
  \quad \text{for all $\talpha \in \Wast$ and all $w_\talpha,v_\talpha \in \HH^1_\talpha(\Omega_H^\Lambda)$.} \label{eq:varFormXXX}
\end{multline}
Extending $w(\talpha, \cdot)$ from~\eqref{eq:heAlpha} by $\Lambda$-periodicity to $\Omega_H$ and setting $\hat{w}(\talpha,\j) = \dhat{u}(\LambdaAst \j+\talpha, H)$ allows to extend $w(\talpha, \cdot)$ to $\Omega$ by
\begin{equation}\label{eq:URCAlpha}
  w(\talpha, \x) = 
  |\det \Lambda|^{-1/2} \, \sum_{\j\in\Z^2} \hat{w}(\talpha,\j) \, e^{\i (\LambdaAst \j + \talpha)\cdot \tx + \i \beta_\j(\talpha, k)\,  (\x_3-H) } \, 
  \quad \text{ for $\x_3 > H$.}
\end{equation}
This extension yields an $\talpha$-quasiperiodic weak solution to $\Delta_\x w + k^2 w =0$ in $\Omega$ that satisfies $(\partial w / \partial \bm{\nu}) - \theta_\Lambda w = \J_{\Gamma} f(\talpha,\cdot)$ on $\Gamma$.
\end{theorem}
\begin{proof}
As in~\cite{Chand2010} for the Dirichlet problem one shows that~\eqref{eq:varFormHEScal} possesses a unique solution in $\H^1_r(\Omega_H)$ if the right-hand side $f$ belongs to $\HH^{-1/2}_r(\Gamma)$ for $|r|<1$. 
Hence, the Bloch transform $w = \J_\Omega u$ belongs to $\HH^r_\0(\Wast; \HH^1_\talpha(\Omega_H^\Lambda))$ by Theorem~\ref{th:BlochOmega} and it merely remains to show that $w$ satisfies~\eqref{eq:heAlpha}. 
To this end, we apply the composition of the inverse Bloch transform and the Bloch transform $\J_{\Omega}$ to a solution $u \in \H^1_r(\Omega_H^\Lambda)$ of~\eqref{eq:varFormHEScal} and invert, i.e., adjunct, $J_{\Omega}^{-1} = \J_{\Omega}^\ast$. 
The gradient $\nabla u$ transforms under the Bloch transform to $\nabla_\x \J_{\Omega} u(\talpha,\cdot) = \nabla_\x w(\talpha,\cdot)$ due to Lemma~\ref{th:BlochPartDeri}. 
As the boundary datum on the right of~\eqref{eq:heAlpha} follows by taking the Bloch transform $\J_{\Gamma}$ of both sides of~\eqref{eq:bcHe}, it merely remains to show that $w(\talpha,\cdot) = \J_{\Omega} u$ satisfies $\partial w (\talpha,\cdot) / \partial \x_3 = T_\talpha^+ w(\talpha,\cdot)$ on $\Gamma_H^\Lambda$ to conclude that $w =\J_{\Omega} u$ satisfies the variational formulation~\eqref{eq:heAlpha}. 
To this end, we show that $w(\talpha,\cdot)$ satisfies~\eqref{eq:URCAlpha}, which implies~\eqref{eq:heAlpha}. 
The radiation condition~\eqref{eq:URCAlpha} for $w(\talpha,\cdot)$ follows from computing the Bloch transform of the upwards radiation condition~\eqref{eq:URC} for $u$ by formula~\eqref{eq:blochHoppla}: 
For $\talpha \in \Wast$ and $\x \in \Gamma_H$, 
\begin{align*}
  \J_{\R^2} \big( u|_{\Gamma_H} \big)(\talpha,x)
  & = \frac1{|\det \Lambda|^{1/2}} \sum_{\j\in\Z^2} \dhat{u}(\LambdaAst \j+\talpha) \, e^{\i (\LambdaAst \j+\talpha) \cdot \tx + \i \sqrt{k^2 - |\LambdaAst \j+\talpha|^2}(\x_3-H)}.
\end{align*}
The reverse direction of the claimed equivalence follows similarly, as the inverse Bloch transform is an isomorphism from $\HH^r_\0(\Wast; \HH^1_\talpha(\Omega_H^\Lambda))$ onto $\HH^1_{r}(\Omega_H)$, see Theorem~\ref{th:BlochOmega}. 
\end{proof}

\begin{remark}
  For $r \leq -1$ it is not obvious how to set up a variational formulation as $T^+$ fails to be continuous between $\HH^{\pm1/2}_{r}(\Gamma_H)$.
  As the Bloch transform is not concerned by this, one might of course take the (formally) transformed problems~\eqref{eq:heAlpha} to define a radiation condition for solutions in, e.g., $\H^1_r(\Omega_H)$ of the Helmholtz equation. 
\end{remark}

We briefly recall well-known solution theory for the decoupled $\talpha$-quasiperiodic problems defined via the sesquilinear forms $a_{k,\talpha}$ from~\eqref{eq:varFormXXX}, assuming from now on that $\Im \theta_\Lambda \geq 0$ on $\Gamma$.
The sesquilinear form $a_{k,\talpha}$ satisfies a G\r{a}rding inequality on $H^1_\talpha(\Omega_H)$, see, e.g.,~\cite{Bonne1994, Elsch1998}, such that existence of solution to the variational problem $a_{k,\talpha}(w_\talpha,v_\talpha)=G(\overline{v_\talpha})$ for all $v_\talpha \in \H^1_\talpha(\Omega_H^\Lambda)$ for fixed $\talpha \in \Wast$ follows from uniqueness. 
Uniqueness for fixed $\talpha$, in turn, can be shown either via particular Rellich identities in case that $\Gamma$ is graph of a function and $\Re\theta_\Lambda \leq 0$ as in~\cite{Chand2005, Chand2010}.

\begin{remark}
One cannot expect $a_{k,\talpha}$ to satisfy more than a Fredholm property, because $x\mapsto \exp(\pm \i k\,  \x_{1,2})$ are solutions to a homogeneous surface scattering problem from the plane $\{ \x_3 = 0 \}$ with Neumann boundary conditions.
%
\end{remark}
 
%
%

We finally note a simple uniqueness result for $\theta_\Lambda$ with positive imaginary part. 

\begin{lemma}\label{th:uniqueAlpha}
If $\theta_\Lambda \in L^\infty(\Gamma)$ satisfies $\Im \theta_\Lambda >0$ on an open, non-empty subset of $\Gamma_\Lambda$, then~\eqref{eq:heAlpha} is uniquely solvable for all $(\talpha,k) \in \Wast\times\R_{>0}$.  
\end{lemma}
\begin{proof}
Computing the imaginary part of $a_{k,\talpha}(v,v)$ for some solution $v$ to the homogeneous problem corresponding to~\eqref{eq:heAlpha}, i.e., for $\J_\Gamma f (\talpha,\cdot) = 0$, shows that the Cauchy data of $v$ vanishes on $\Gamma_\Lambda$. 
Thus, Holmgren's lemma implies the claim. 
\end{proof}

\section{Regularity and decay estimates}\label{se:regDec}
In this section, we show a regularity result for the solution $w$ to the periodic problem~\eqref{eq:heAlpha} in $\talpha$ that yields decay of the solution to a scattering problem~\eqref{eq:varFormHEScal} for particular incident Herglotz wave functions.
To this end, let us introduce the spaces $\WW^{1,p}_\0(\Wast; \HH^1_\talpha(\Omega_H^\Lambda))$ for $1 \leq p < \infty$, as described in Remark~\ref{rem:X} as the space of those distributions in $\mathcal{D}'(\R^2 \times \Omega_H)$ that are $\LambdaAst$-periodic in their first and $\talpha$-quasiperiodic in their second argument, and possess finite norm
\[
  \| w \|_{\WW^{1,p}_\0(\Wast; \HH^1_\talpha(\Omega_H^\Lambda))}
  := \bigg[ \int_{\Wast} \hspace*{-2pt} \bigg[ \| w(\talpha,\cdot) \|_{\HH^1_\talpha(\Omega_H^\Lambda))}^p  + \sum_{j=1,2} \| \partial_{\talpha_j} w(\talpha,\cdot) \|_{\HH^1_\talpha(\Omega_H^\Lambda))}^p \bigg] \hspace*{-2pt} \d{\talpha} \bigg]^{1/p} < \infty.
\]

\begin{theorem} \label{th:exiSolScalPerio}
Assume that the problem to find $w_\talpha \in \H^1_\alpha(\Omega_H^\Lambda)$ with 
\begin{equation} \label{eq:uniquenessAlpha}
  a_{k,\talpha}(w_\talpha,v_\talpha)=0 \qquad \text{for all $v_\talpha \in \H^1_\alpha(\Omega_H^\Lambda)$}
\end{equation}
is merely solved by the trivial solution $w_\talpha = 0$ for each $\talpha \in \Wast$. 

(a) If $f \in \HH^{-1/2}(\Gamma)$, then the solution $w$ to~\eqref{eq:heAlpha} belongs to $L^2(\Wast; \HH^1_\talpha(\Omega_H^\Lambda))$ and the solution $u = \J_{\Omega}^{-1} w$ to~\eqref{eq:varFormHEScal} belongs to $\HH^{1}(\Omega_H^\Lambda)$. 

(b) If $f \in \HH^{-1/2}(\Gamma)$ satisfies  $\J_{\Gamma} f \in \WW^{1,p}_\0(\Wast; \H^{-1/2}_\talpha(\Gamma_\Lambda))$ as well as $\sup_{\talpha \in \Wast} \| \J_{\Gamma} f(\talpha, \cdot) \|_{\H^{-1/2}_\talpha(\Gamma_\Lambda)} < \infty$, then the solution $w$ to~\eqref{eq:heAlpha} belongs to $\WW^{1,p}_\0(\Wast; \HH^1_\talpha(\Omega_H^\Lambda))$ for $1\leq p < 2$.
\end{theorem}
\begin{remark}\label{th:chandler}
(a) If $f \in \HH^{-1/2}_r(\Gamma)$ with $r>1$, then $\mathcal{J}f \in \HH^r_\0(\Wast; \H^{-1/2}_\talpha(\Gamma_\Lambda))$ such that Sobolev's embedding theorem states that $\talpha \mapsto \mathcal{J}f(\talpha,\cdot)$ is continuous, and in particular bounded, such that $\max_{\talpha \in \Wast} \| \mathcal{J}f(\talpha,\cdot) \|_{\HH^{-1/2}_\talpha(\Gamma_\Lambda)} < \infty$ is finite. 
In the end of this section, we state an example for such a boundary term due to scattering of a Herglotz wave from~$\Gamma$.  

  
(b) The proof of Theorem~\ref{th:exiSolScalPerio}(b) indicates that the claimed range in $p$ is strict. 
This relates to a result due to Chandler-Wilde and Elschner from~\cite{Chand2010}, showing that the Dirichlet problem corresponding to~\eqref{eq:varFormHEScal} generally fails to possess unique solutions in $\HH^{1}_r(\Omega_H^\Lambda)$ for boundary data in $\HH^{1/2}_r(\Gamma_\Lambda)$ with $|r|\geq 1$.
\end{remark}
\begin{proof}
The sesquilinear form $a_{k,\talpha}$ is independent of $\talpha$ except for the term 
\[
  (w(\talpha,\cdot),v) \mapsto \int_{\Gamma_H} T_\talpha^+ w(\talpha, \cdot) \, \overline{v} \dS 
  = \i \sum_{\j\in\Z^2} \sqrt{k^2 - |\LambdaAst \j +\talpha|^2} \, \hat{w}(\talpha,\j) \overline{\hat{v}(\j)},
\]  
which is, however, continuous in $\talpha$ as we show now:  
For $\j\in\Z^2$ such that $|\LambdaAst \j + \talpha| \geq 2k$, the square roots $\talpha \mapsto \sqrt{k^2 - |\talpha + \LambdaAst \j |^2}$ are infinitely smooth functions on $\R^2$.  
For the remaining finitely many $\j$, the roots $\talpha \mapsto \sqrt{k^2 - |\LambdaAst \j + \talpha|^2}$ are all continuous on $\R^2$, such that the convergence of the series defining $T_\talpha^+$ in the operator norm of $\mathcal{L}(\HH^{1/2}_\talpha(\Gamma_H^\Lambda),\HH^{-1/2}_\talpha(\Gamma_H^\Lambda))$ ensures that $(w(\talpha,\cdot),v) \mapsto (T_\talpha^+ w(\talpha, \cdot),v)$ and $a_{k,\talpha}$ both depend continuously on $\talpha \in \Wast$.  
The continuous operator $A_{k,\talpha}: \, \HH^1_\talpha(\Omega_H^\Lambda) \to \big[ \HH^1_\talpha(\Omega_H^\Lambda)  \big]^\ast$, associated to $a_{k,\talpha}$ by $\langle A_{k,\talpha} w, v \rangle = a_{k,\talpha} (w,v)$ for all $v,w \in \HH^1_\talpha(\Omega_H^\Lambda)$, is hence continuous in $\talpha \in \Wast$. 
(Here $\langle \cdot \,,\, \cdot \rangle$ denotes the anti-linear duality product between $\HH^1_\talpha(\Omega_H^\Lambda)$ and its dual.)
Due to Fredholm theory, our assumption that~\eqref{eq:uniquenessAlpha} is uniquely solvable for all $\talpha \in \Wast$ implies that $A_{k,\talpha}$ is invertible for all $\talpha \in W^\ast$.  
Both the operator and its inverse are moreover $\LambdaAst$-periodic in $\talpha \in \R^2$. 
As $A_{k,\talpha}$ depends continuously on $\talpha$, its inverse $A_{k,\talpha}^{-1}$ is continuous in $\talpha$, too. 
In particular, the operator norms $\| A_{k,\talpha}^{-1} \|$ are uniformly bounded in $\talpha \in W^\ast$.  

(a) The uniform bound for the inverses $\| A_{k,\talpha}^{-1} \|$ implies for the solution $w$ to~\eqref{eq:heAlpha} that 
\begin{align}
  \| w \|_{L^2(\W^\ast; \HH^1_\talpha(\Omega_H^\Lambda))}^2  
  & = \int_\Wast \| w(\talpha,\cdot) \|_{\HH^1_\talpha(\Omega_H^\Lambda)}^2 \d{\talpha}
  \leq C \int_\Wast \| \J_{\Gamma} f(\talpha,\cdot) \|_{\HH^{-1/2}_\talpha(\Gamma_\Lambda)}^2 \d{\talpha} \nonumber \\
  & = C \| \J_\Gamma f \|_{L^2(\Wast; \H^{-1/2}_\talpha(\Gamma_\Lambda))}^2 
  = C \| f \|_{\H^{-1/2}(\Gamma)}^2.   \label{eq:aux1088}
\end{align}
(b) 
For simplicity, we abbreviate in this part $\HH^{-1/2}_\talpha(\Gamma_\Lambda)$ by $\HH^{-1/2}_\talpha$. 
As in~\eqref{eq:aux1088}, the uniform bound for the inverses $A_{k,\talpha}^{-1}$ implies by H\"older's inequality that $\| w \|_{\WW^{1,p}_\0(\Wast; \HH^1_\talpha(\Omega_H^\Lambda)} \leq C(p) \| \J_{\Gamma} f \|_{L^2(\Wast; \HH^{-1/2}_\talpha)}$. 
Further, the extension of $w(\talpha, \cdot)$ to $\Omega$, see~\eqref{eq:URCAlpha}, solves the Helmholtz equation with constant coefficients in $\Omega$, such that standard elliptic regularity results imply that $w(\talpha, \cdot)|_{\Gamma_H^\Lambda} \in \HH^s_\talpha(\Gamma_H^\Lambda)$ for arbitrary $s\geq 0$.  
Additionally, for all $s \geq 0$ there is $C(s)>0$ such that $\sum_{\j\in\Z^2} (1+|\j|^2)^s | \hat{w}(\talpha, \j)|^2 \leq C(s) \| w(\talpha,\cdot)\|_{\HH^1_\alpha(\Omega_H^\Lambda)}^2$ (where $\hat{w}(\talpha, \j)$ again denotes the $\j$th Fourier coefficient of $w(\talpha, \cdot)|_{\Gamma_H^\Lambda}$, see~\eqref{eq:URCAlpha}). 
In consequence, $| \hat{w}(\talpha, \j)|^2 \leq C \, (1+|\j|^2)^{-2}  \| w(\talpha,\cdot)\|_{\HH^1_\alpha(\Omega_H^\Lambda)}^2$ for all $\j\in\Z^2$ and $\talpha \in \Wast$. 

The solution $w$ to~\eqref{eq:heAlpha} possesses distributional derivatives $w'_{1,2} = \partial_{\talpha_{1,2}} w(\talpha,\cdot)$ with respect to $\talpha_{1,2}$ that satisfy $a_{k,\talpha} (w'_{1,2},v) = F_{\talpha_{1,2}} (v)$ for all $v\in \H^1_\talpha(\Omega_H^\Lambda)$, where  
\[  
  F_{\talpha_{1,2}} (v) 
  = \int_{\Gamma_\Lambda} \partial_{\talpha_{1,2}} \big(\J_{\Gamma} f \big)(\talpha,\cdot) \, \overline{v} \dS 
  + \i \sum_{\j\in\Z^2} \frac{(\LambdaAst \j + \talpha)_{1,2}}{\sqrt{k^2 - |\LambdaAst\j + \talpha |^2}} \, \hat{w}(\talpha,\j) \overline{\hat{v}(\j)}. \label{eq:RHSDeriAlpha}
\]
(Again, $\hat{v}(\j)$ denotes the $\j$th Fourier coefficient of the restriction of $v$ to $\Gamma_H^\Lambda$.)
This right-hand side is well-defined and bounded for $\talpha \in \Wast \setminus E_\LambdaAst$ where $E_\LambdaAst = \{ \talpha \in \Wast: \, |\LambdaAst\j + \talpha | = k$ for some $\j \in \Z^2 \}$.
Note that the equation $|\LambdaAst\j + \talpha | = k$ can only be satisfied for finitely many $\j \in\Z^2$, e.g., for those in $I = \{ \j\in\Z^2: \, |\LambdaAst \j| \leq k + |\LambdaAst| \}$.
The set $E_\LambdaAst$ hence consists of finitely many one-dimensional hypersurfaces in $\Wast$ of finite length. 
As solutions $\bm\xi \in \R^2$ to $|\LambdaAst\j + \bm\xi | = k$ lie on a circle centered at $\LambdaAst\j$, the set $E_\LambdaAst$ moreover is a union of smooth curves that are contained in a finite union of circles. 
As 
\[
  \frac{\big| (\LambdaAst \j + \talpha)_{1,2} \big|}{| k^2 - |\LambdaAst \j + \talpha|^2 |^{1/2}} 
  \leq 
  \frac{|\LambdaAst \j + \talpha|}{| k^2 - |\LambdaAst \j + \talpha|^2 |^{1/2}} 
  \to 1  
  \quad \text{for } |\j| \to \infty 
\]
holds for all $\talpha \in \Wast \setminus E_{\LambdaAst}$, we conclude for all $v\in \HH^1_\talpha(\Omega_H^\Lambda)$ with $\| v \|_{\HH^1_\talpha(\Omega_H^\Lambda)} = 1$ that  
\begin{align*}
  \big| & F_{\talpha_{1,2}} (v) \big|
  \leq 
  \| \partial_{\talpha_{1,2}} \J_{\Gamma} f (\talpha,\cdot) \|_{\HH^{-1/2}_\talpha} +  \sum_{\j\in\Z^2} \frac{|\LambdaAst \j + \talpha| \, (1+|\j|^2)^{-2} \| w(\talpha,\cdot) \|_{\HH^1_\talpha(\Omega_H^\Lambda)}}{| k^2 - |\LambdaAst \j + \talpha|^2 |^{1/2}}  \d{\talpha}   \\
   & \leq \| \partial_{\talpha_{1,2}} \J_{\Gamma} f (\talpha,\cdot) \|_{\HH^{-1/2}_\talpha}  + \bigg( \sum_{\j\in I} 
   +   \sum_{\j\in\Z^2\setminus I} \bigg) \frac{|\LambdaAst \j + \talpha| \, (1+|\j|^2)^{-2} }{| k^2 - |\LambdaAst \j + \talpha|^2 |^{1/2}} 
  \| \J_{\Gamma} f (\talpha,\cdot) \|_{\HH^{-1/2}_\talpha}  \\
  & \leq \| \partial_{\talpha_{1,2}} \J_{\Gamma} f (\talpha,\cdot) \|_{\HH^{-1/2}_\talpha}  + \bigg(\sum_{\j\in I} \frac{C_1(\Lambda,k)}{| k^2 - |\LambdaAst \j + \talpha|^2 |^{1/2}}
   + C_2(\Lambda,k) \bigg)\|  \J_{\Gamma} f (\talpha,\cdot) \|_{\HH^{-1/2}_\talpha}
\end{align*}
with constants $C_{1,2}(\Lambda,k)$ independent of $\talpha \in\Wast \setminus E_\LambdaAst$ and $f$. 
As the right-hand side $F_{\talpha_{1,2}}$ is well-defined for all $\talpha \in \Wast \setminus E_\LambdaAst$, and as the solutions operators $A_{k,\talpha}^{-1}$ are uniformly bounded in $\talpha \in \Wast$, there holds $\| w'_{1,2}(\talpha,\cdot) \|_{\HH^1_\talpha(\Omega_H^\Lambda)} \leq C \| F_{\talpha_{1,2}} \|_{\HH^1_\talpha(\Omega_H^\Lambda)'}$ for all $\talpha \in \Wast \setminus E_\LambdaAst$.
By Jensen's inequality, $(a_1+ \dots+a_n)^p \leq n^{p-1} (a_1^p+ \dots+a_n^p)$ holds for $a_1, \dots ,a_n \geq 0$ and shows that 
\begin{align}
  &\| w'_{1,2}   \|_{L^p(\Wast; \HH^1_\talpha(\Omega_H^\Lambda))}^p  
  = \int_\Wast \| w'_{1,2}(\talpha, \cdot ) \|_{\HH^1_\talpha(\Omega_H^\Lambda)}^p \d{\talpha} 
  \leq C \int_\Wast \big| F_{\talpha_{1,2}}\big|_{[\HH^1_\talpha(\Omega_H^\Lambda)]^\ast}^p \d{\talpha} \label{eq:aux1146} \\
  & \leq C \left[  \| \J_{\Gamma} f \|_{\WW^{1,p}_\0(\Wast; \HH^1_\talpha(\Omega_H^\Lambda))}^p 
  + \int_\Wast \bigg[ \sum_{\j\in I} \frac{C_1(\Lambda,k)^p}{| k^2 - |\LambdaAst \j + \talpha|^2 |^{p/2}}  \bigg]
  \| \J_{\Gamma} f (\talpha,\cdot) \|_{\HH^{-1/2}_\talpha}^p \d{\talpha} \right]  \nonumber \\
   & \leq C \left[ \| \J_{\Gamma} f \|_{\WW^{1,p}_\0(\Wast; \HH^1_\talpha(\Omega_H^\Lambda))}^p
    + \sum_{\j\in I} \int_\Wast \frac{1}{| k^2 - |\LambdaAst \j + \talpha|^2 |^{p/2}} \d{\talpha}  \right] \hspace*{-1pt} \sup_{\talpha \in \Wast} \| \J_{\Gamma}f(\talpha,\cdot) \|_{\HH^{-1/2}_\talpha}^p. \nonumber
\end{align}
(Note that all series have merely finitely many terms, such that the above inequality can indeed be applied.)
The integrand in the last expression is singular at all points in $E_\LambdaAst$, which contains finitely many pieces of one-dimensional circles. 
The singularity of 
\[
  \talpha \mapsto \frac{1}{| k^2 - |\LambdaAst \j + \talpha|^2 |^{p/2}}
  = \frac{1}{[k+|\LambdaAst \j + \talpha|]^{p/2}} \frac{1}{|k-|\LambdaAst \j + \talpha||^{p/2}}, 
  \quad
  \talpha \in \Wast \setminus E_{\LambdaAst},
\]
is of the order $p/2$, which is strictly less than one as $1\leq p < 2$. 
Each hypersurface $E_\LambdaAst^{(\j,\delta)} = \{ \bm\xi \in \Wast : \, |\LambdaAst \j + \bm\xi| = k+\delta \}$ for $\delta \geq 0$ and $\j$ such that $|\LambdaAst \j| \leq 2k + |\LambdaAst|$ is at least piecewise Lipschitz smooth and possesses finite surface area bounded by some constant times the surface area of $E_\LambdaAst$. 
As the union of all the hypersurfaces over $0 \leq \delta \leq R_0$ and $\j$ such that $|\LambdaAst \j| \leq 2k + |\LambdaAst|$ covers $\Wast$ if $R_0>0$ is chosen large enough, we estimate that 
\begin{align*} 
  \int_\Wast \frac{1}{| k - |\LambdaAst \j + \talpha| |^{p/2}} \d{\talpha}
  & \leq \sum_{\j\in I} \int_0^{R_0} \int_{E_\LambdaAst^{(\j,\delta)}} \delta^{-p/2} \dS{\talpha} \d{\delta} 
  = \sum_{\j\in I} \big| E_\LambdaAst^{(\j,\delta)} \big| \frac{R_0^{1-p/2}}{1-p/2}  < \infty.
\end{align*}
Thus, the integral $\int_\Wast | k^2 - |\LambdaAst \j + \talpha|^2 |^{-p/2} \d{\talpha}$ from~\eqref{eq:aux1146} takes a finite value, such that $w'_{1,2} \in L^p(\Wast; \HH^1_\talpha(\Omega_H^\Lambda))$ and $w \in \WW^{1,p}_\0(\Wast; \HH^1_\talpha(\Omega_H^\Lambda))$.
\end{proof}

Estimate~\eqref{eq:aux1146} in the last proof shows that a sufficient condition for the $w_{1,2}'$ to be square-integrable in $\talpha$ is that $\talpha \mapsto J_\Gamma f$ vanishes up to sufficiently high order for $\talpha \in E_\LambdaAst$. 

\begin{corollary}\label{th:coconut}
Suppose that the assumptions of Theorem~\ref{th:exiSolScalPerio}(b) are satisfied.
If, additionally, $\talpha \mapsto \J_\Gamma f(\talpha,\cdot)$ vanishes in a neighborhood of the $\LambdaAst$-periodic set $\{ \bm\xi \in \R^2: \, \exists \, \j \in \Z^2 \text{ s.th. } |\LambdaAst \j + \bm\xi | = k \}$, then $w$ belongs to $\HH^1_\0(\Wast; \HH^1_\talpha(\Omega_H^\Lambda))$ and the solution $u$ to the surface scattering problem~\eqref{eq:varFormHEScal} belongs to $\HH^1_1(\Omega_H)$.   
\end{corollary}

The following example shows that the assumptions of Corollary~\ref{th:coconut} are satisfied for incident Herglotz wave functions $v_g$ if that the plane wave representation of $v_g$ does not incorporate plane waves traveling in directions orthogonal or nearly orthogonal to $\e_3$. 
 
\begin{example}\label{ex:1}
For any continuously differentiable function $g$ with compact support on the lower half-sphere, the Bloch transform of an incident Herglotz wave function $v_g$ is continuously differentiable in $\talpha \in \R^2$ and infinitely smooth in $\tx$, such that the corresponding solution $u$ to the scattering problem~\eqref{eq:varFormHEScal} belongs to $\HH^1_1(\Omega_H)$.   
By definition, 
\[
  v_g(\x) := \int_{\S^2_-} e^{\i k \, \x \cdot \bm\theta} g(\bm\theta) \dS{(\bm\theta)}
   \quad \text{for } \x \in \overline\Omega,
\]
where $\S^2_- := \{ \bm\theta \in \S^2: \, \bm\theta_3 < 0\}$ is the lower unit sphere, such that all plane waves in the representation of $v_g$ are propagating downwards, and $g: \, \S^2_- \to \C$. 
It is well-known that $v_g$ is an entire solution to the Helmholtz equation. 
The diffeomorphism $\Xi: \, \tilde\theta \mapsto (\tilde\theta, \, (1-|\tilde\theta|^2)^{-1/2})^\top$ from the unit disc $D =\{ \tilde{\bm\theta} \in \R^2: \, |\tilde{\bm\theta}| <1 \}$ onto $\S^2_-$ yields
\[
  v_g(\x)
  = \int_{D} e^{\i k \big[ \tx \cdot \tilde{\bm\theta} - \x_3 / [ 1-|\tilde{\bm\theta}|^2 ]^{1/2} \big]} \frac{g\big(\Xi(\tilde{\bm\theta})\big)}{\big[ 1-|\tilde{\bm\theta}|^2 \big]^{1/2}} \d{\tilde{\bm\theta}}, 
  \quad \x \in \overline\Omega. 
\]
The Bloch transform of $v_g$ hence formally equals
\[  
  \left[\J_\Omega v_g \right](\talpha, \x) 
  =  \frac{|\det \Lambda|}{2\pi}^{1/2} \sum_{\j\in\Z^2}  \int_{D} e^{\i \, \Lambda \j \cdot (k \tilde{\bm\theta} -\talpha)} e^{\i k \big[\tx \cdot \tilde{\bm\theta} - \x_3 / [ 1-|\tilde{\bm\theta}|^2 ]^{1/2} \, \big]} \frac{g\big(\Xi(\tilde{\bm\theta})\big)}{\big[ 1-|\tilde{\bm\theta}|^2 \big]^{1/2}} \d{\tilde{\bm\theta}}
\]
for $\talpha \in \Wast$ and $\x \in \overline\Omega$. 
As in~\cite[Lemma 5]{Lechl2015e} one shows that $\sum_{\j\in\Z^2} e^{\i \, \Lambda \j \cdot \bm\mu} = 2\pi/|\det \Lambda|^{-1/2} \, \delta_\LambdaAst(\bm\mu) \in \mathcal{D}'_\0(\R^2)$ is the $\LambdaAst$-periodic Dirac distribution at the origin, i.e., $\langle \delta_\LambdaAst , \psi \rangle = \sum_{\bell\in\Z^2} \psi(\LambdaAst\bell)$.  
The arguments proving Theorems 6 and 7 in~\cite{Lechl2015e} then show that the Bloch transform of $v_g$ equals 
\begin{align*}
  \left[\J_\Omega v_g \right](\talpha, \x) 
  & =  k  \sum_{\bell\in\Z^2: \, (\LambdaAst \bell+\talpha)/k \in D} e^{\i \big[ \tx \cdot (\LambdaAst \bell+\talpha) - \x_3 / [ k^2-|(\LambdaAst \bell+\talpha)|^2 ]^{1/2}) \big]} \frac{g\big(\Xi((\LambdaAst \bell+\talpha)/k)\big)}{\big[ k^2-|(\LambdaAst \bell+\talpha)|^2 \big]^{1/2}} , 
\end{align*}
for any $g \in C^\infty_0(\S^2_-)$. 
Indeed, the denominator of the last fraction vanishes if and only if $k=|\LambdaAst \bell+\talpha|$, that is, if and only if $|(\Lambda \bell+\talpha)|/k=1$, such that $\Xi((\LambdaAst \bell+\talpha)/k)$ belongs to the boundary of $\S^2_-$ where $g$ vanishes by its compact support.
Further, the finitely many series indices in $\{ \bell\in\Z^2: \, (\LambdaAst \bell+\talpha)/k \in D \}$ can only change at some $\talpha$ if $|\LambdaAst \bell +\talpha| = k$, such that $g \circ \Xi$ vanishes in a neighborhood of $\talpha$. 
Compactness of the support of $g$ hence implies that the regularity of $g$ transfers to the regularity of $\J_\Omega v_g$ in $\talpha$. 
(The latter function is anyway smooth in $\x$.)
Thus, for any continuously differentiable function with compact support in $\S^2_-$, the Bloch transform of $\J_\Omega v_g$ is continuously differentiable in $\talpha \in \R^2$ and infinitely smooth in $\x$. 
\end{example}

\section{Scattering from Locally Perturbed Periodic Surfaces}\label{se:pertScatt}
In this final section we analyze wave scattering from a locally perturbed periodic surface 
\[
  \Gamma_\s := \big\{ {\bm\zeta}_\s(\tilde{\y},0): \, \tilde{\y} \in \R^2 \big\}
\]
that is defined via a diffeomorphism $\bm\zeta_\s: \, \R^3 \to\R^3$ that equals $\zeta$ outside $\Omega_{H_0}^\Lambda$, i.e., ${\bm\zeta}_\s = {\bm\zeta}$ holds in $\R^3 \setminus \Omega_{H_0}^\Lambda$. 
The periodic surface $\Gamma = \big\{ {\bm\zeta}(\tilde{\y},0): \, \tilde{\y} \in \R^2 \big\}$ hence coincides with  $\Gamma_\s$ outside $\Omega_{H_0}^\Lambda$ (that is, for pre-images $\tilde{\y} \not \in \W$). 
Despite we have chosen the periodicity cell $\Omega_{H_0}^\Lambda$ as domain of perturbation, one can of course always consider perturbations in regions domain by replacing the periodicity matrix $\Lambda$ by $m\, \Lambda$ for some $m\in\N$.  
In analogy to $\Omega$ and $\Omega_H$, we introduce two domains 
\[
  \Omega_\s := \big\{ {\bm\zeta}_\s(\y): \, y \in \R^3, \y_3 >0 \big\}  
  \quad \text{and } \quad 
  \Omega_\s^H := \big\{ \x'\in \Omega_\s: \, \x_3'<H \big\} \quad \text{for some } H \geq H_0. 
\]
(Points in $\Omega_\s$ and $\Omega_\s^H$ will be denoted by $\x'$.)
Note that $\Psi_\s: \, \Omega \to \Omega_\s$ defined by $\Psi_\s = {\bm\zeta}_\s \circ {\bm\zeta}^{-1}$ is a Lipschitz diffeomorphism mapping $\Omega$ onto $\Omega_\s$. 
We denote its inverse by $\Phi_\s := {\bm\zeta} \circ {\bm\zeta}_\s^{-1}: \, \Omega_\s \to \Omega$ and note that the support of $\Psi_\s - I_3: \, \Omega \to \R^3$ is contained in the bounded domain $\Omega_H^\Lambda$.
In particular, $\Psi_\s: \, \Omega_H \to \Omega_\s^H$ and $\Phi_\s : \, \Omega_\s^H \to \Omega_H$ are isomorphisms for any $H\geq H_0$ that are inverses to each other. 

We introduce a boundary source $f \in \HH^{-1/2}_r(\Gamma_\s)$ for some $r>-1$ and a coefficient $\theta \in L^\infty(\Gamma_\s)$ such that $\theta$ equals the periodic Robin coefficient $\theta_\Lambda$ from Section~\ref{se:perioSurfScatt} on $\Gamma \setminus \Gamma_\Lambda$. 
For these boundary functions, we consider a locally perturbed periodic scattering problem with Robin or impedance boundary condition, 
\begin{equation}\label{eq:scalPert}
  \Delta u + k^2 u = 0 \quad \text{in } \Omega_\s^H, \quad 
  \frac{\partial u}{\partial\bm{\nu}} - \theta \, u = f \quad \text{on } \Gamma_\s, \quad 
  \frac{\partial u}{\partial\x_3'} = T^+ \big( \left. u \right|_{\Gamma_H}\big) \quad \text{on } \Gamma_H.
\end{equation}
(Again, the normal $\bm\nu$ on $\Gamma$ points into $\Omega$.) 
As in Section~\ref{se:perioSurfScatt}, the variational formulation of~\eqref{eq:scalPert} is to find $u\in \HH^1_r(\Omega_\s^H)$ solving 
\begin{equation}\label{eq:varFormPertScal}
  \int_{\Omega_\s^H} \left[ \nabla u \cdot \nabla \overline{v} - k^2 u \overline{v} \right] \d{\x'}
  + \int_{\Gamma_\s} \theta \, u \overline{v} \dS  
  - \int_{\Gamma_H} T^+(u) \overline{v} \dS  
  = \int_{\Gamma_\s} f \overline{v} \dS 
\end{equation}
for all $v \in \HH^1(\Omega_\s^H)$ with compact support included in $\overline{\Omega_H}$.
(As for~\eqref{eq:varFormHEScal}, this formulation is well-defined for $r>-1$.)   
The latter problem is by the transformation theorem equivalent to finding a solution $\x \mapsto u_\s(\x) = u \circ \Psi_\s(\x) \in \HH^1_r(\Omega_H)$ of the transformed variational problem 
\begin{equation}\label{eq:varFormPertScalPerio}
  \int_{\Omega_H} \left[ A_\s \nabla u_\s \cdot \nabla  \overline{v_\s} - k^2 \, c_\s \, u_\s \overline{v_\s} \right] \d{\x}
  + \int_{\Gamma} \theta_\s \, u_\s \overline{v_\s} \dS  
  - \int_{\Gamma_H} T^+(u_\s) \overline{v_\s} \dS  
  = \int_{\Gamma} f_\s \overline{v_\s} \dS 
\end{equation}
for all $v_\s \in \HH^1(\Omega_H)$ with compact support. 
By the transformation theorem the coefficients and right-hand side equal     
\begin{align*}
  A_\s(\x) & := |[\det \nabla_{\x'} \Phi_\s]|\circ \Psi_\s(\x) \ \big[ [ \nabla_{\x'} \Phi_\s] [\nabla_{\x'} \Phi_\s]^\top \big] \circ \Psi_\s(\x) \in L^\infty(\Omega_H, \ \R^{3\times 3}), \\
  c_\s(\x) & := |[\det \nabla_{\x'} \Phi_\s]|\circ \Psi_\s(\x) \in L^\infty(\Omega_H), \\ 
  \theta_\s(\x) & := |[\det \nabla_{\x'} \Phi_\s]|\circ \Psi_\s(\x) \ |[\nabla_{\x'} \Phi_\s]^{-\top} \bm{\nu}| \circ \Psi_\s(\x) \ \theta\circ\Psi_\s(\x) \in L^\infty(\Gamma), \quad \text{and} \\  
  f_\s(\x) & := |[\det \nabla_{\x'} \Phi_\s]|\circ \Psi_\s(\x) \ | [\nabla_{\x'} \Phi_\s]^{-\top} \bm{\nu}|\circ\Psi_\s(\x) \ f\circ\Psi_\s(\x) \in \HH^{-1/2}_r(\Gamma).
\end{align*}
%
%
\begin{remark}
The strong formulation of~\eqref{eq:varFormPertScalPerio} reads $\divSpace \left( A_\s \nabla u_\s \right) + k^2 c_\s u_\s = 0$ in $\Omega_H$ and $\bm{\nu} \cdot A_\s \nabla u_\s - \theta_\s u_\s = f_\s$ on $\Gamma$, subject to $\partial u_\s / \partial \x_3 = T^+(u_\s)$ on $\Gamma_H$. 
\end{remark}

We reformulate~\eqref{eq:varFormPertScalPerio} by applying the inverse Bloch transform composed with the Bloch transform to the weak solution $u_\s$ and afterwards exploit that $\J_\Omega^{-1}$ and $J_\Gamma^-1$ equal the $L^2$-adjoints $\J_\Omega^\ast$ and $J_\Gamma^\ast$, respectively. 
As $A_\s-I_3$ is supported in $\Omega_H^\Lambda$, the Bloch transform of $(A_\s-I_3) \nabla u \in L^2(\Omega_H)$ equals ${|\det \Lambda|}/({2\pi}^{1/2}) \, (A_\s-I_3) \nabla u$ in $L^2_\talpha(\Omega_H^\Lambda)^d$, such that it simplifies expressions if we once add and subtract the terms $\nabla u_\s \cdot \nabla \overline{v_\s}$ and $u_\s \, \overline{v_\s}$ in the first two integrals on the left of~\eqref{eq:varFormPertScalPerio}.  
The strong formulation of the boundary value problem hence reads   
\[
  \Delta_\x w_\s(\talpha,\cdot) + k^2 w_\s(\talpha, \cdot) 
  + \frac{|\det \Lambda|}{2\pi}^{1/2} \left[ \divSpace \big((A_\s-I_3) \, \nabla \J_{\Omega}^{-1} w_\s \big) + k^2 (c_\s - 1) \J_{\Omega}^{-1} w_\s \right] = 0 
\]
in $\Wast \times \Omega_H^\Lambda$, augmented by the boundary- and radiation conditions
\[ 
  \frac{\partial w_\s(\talpha, \cdot)}{\partial \bm{\nu}} + \theta_\Lambda \, w_\s(\talpha, \cdot) + \frac{|\det \Lambda|}{2\pi}^{1/2} (\theta_\s-\theta_\Lambda)  \J_{\Omega}^{-1} w (\talpha, \cdot) 
  = \J_{\Gamma} f_\s(\talpha, \cdot)  \quad \text{in } \Wast \times \Gamma_\Lambda, 
\]
and $(\partial w_\s / \partial \x_3) (\talpha, \cdot) = T_\talpha^+ \big( w_\s(\talpha,\cdot) |_{\Gamma_H^\Lambda}\big)$ on $\Wast \times \Gamma_H^\Lambda$.  
Recalling the sesquilinear form $a_{k,\talpha}$ from~\eqref{eq:heAlpha} and the variational formulation~\eqref{eq:varFormPertScalPerio} for $u_\s$, the variational formulation for $w_\s = \J_{\Omega} u_\s$ in $\HH^r_\0(\Wast; \HH^1_\talpha(\Omega_H^\Lambda))$ further reads 
\begin{align}
  & \int_{\Wast} a_{k,\talpha} \big( w_\s(\talpha,\cdot), \, \J_\Omega v_\s(\talpha,\cdot) \big) \d{\talpha}
  - \int_{\Gamma_\Lambda} (\theta_\s-\theta_\Lambda) \, \J_{\Gamma}^{-1} w_\s \, \overline{v_\s} \dS \label{eq:varFormScal} \\ 
  & + \int_{\Omega_H^\Lambda} \left[ (A_\s-I_3) \, \nabla \J_{\Omega}^{-1} w_\s \cdot \nabla \overline{v_\s} + k^2 (c_\s-1) \, \J_{\Omega}^{-1} w_\s\, \overline{v_\s} \right] \d{\x}  
  = \int_{\Wast} \int_{\Gamma_\Lambda} \J_\Gamma f_\s \, \overline{\J_\Gamma v_\s} \dS \d{\talpha} \nonumber 
\end{align} 
for all $v_s \in \HH^1(\Omega_H)$ with compact support in $\overline{\Omega_H}$.  
(See, e.g.,~\cite[Proof of Theorem 4 in Section 6.3]{Evans1998}.)
In contrast to~\eqref{eq:heAlpha}, this problem is coupled in $\talpha$. 
If $r \geq 0$, an equivalent formulation is to find $w_\s \in \HH^r_\0(\Wast; \HH^1_\talpha(\Omega_H^\Lambda))$ such that for almost all $\talpha \in \Wast$ and all $z_\talpha \in \H^1_\talpha(\Omega_H^\Lambda)$ there holds 
\begin{multline}
  a_{k,\talpha} \big( w_\s(\talpha,\cdot), \,z_\talpha \big) 
  - \int_{\Gamma_\Lambda} (\theta_\s-\theta_\Lambda) \, \J_{\Gamma}^{-1} w_\s \, \overline{z_\talpha} \dS \label{eq:varFormScalSimple}\\ 
      + \int_{\Omega_H^\Lambda}  \left[ (A_\s-I_3) \, \nabla \J_{\Omega}^{-1} w_\s \cdot \nabla \overline{z_\talpha} + k^2 (c_\s-1) \, \J_{\Omega}^{-1} w_\s\, \overline{z_\talpha} \right] \d{\x}  
  = \int_{\Gamma_\Lambda} \J_\Gamma f_\s(\talpha,\cdot) \, \overline{z_\talpha} \dS .
\end{multline} 

\begin{theorem}\label{th:scalPertEqui}
Assume that $f_\s$ belongs to $\HH^{-1/2}_r(\Gamma)$ for some $r>-1$. 
Then $u_\s \in \H^1_r(\Omega_\s^H)$ solves~\eqref{eq:varFormPertScalPerio} if and only if $w_\s = \J_{\Omega} u_\s \in \HH^r_\0(\Wast; \HH^1_\talpha(\Omega_H^\Lambda))$ solves~\eqref{eq:varFormScal} (or, alternatively,~\eqref{eq:varFormScalSimple} in case that $r \geq 0$).  
\end{theorem}
\begin{proof}
The arguments from the proof of Theorem~\ref{th:equiScalarPerio} and those before~\eqref{eq:varFormScal} show that  $w_\s = \J_{\Omega} u_\s \in \HH^r_\0(\Wast; \HH^1_\talpha(\Omega_H^\Lambda))$ solves~\eqref{eq:varFormScal}; the reciprocal direction essentially follows by the mapping properties of $\J_{\Omega}$ and $\J_\Gamma$, see Theorem~\ref{th:BlochOmega}.   
\end{proof}

The next theorem shows that~\eqref{eq:varFormScal} is uniquely solvable if the perturbations ${\bm\zeta}_\s-{\bm\zeta}$ and $\theta \circ \Psi_\s - \theta_\Lambda = \theta \circ {\bm\zeta}_\s \circ {\bm\zeta}^{-1} - \theta_\Lambda$ are small enough. 

\begin{theorem}\label{th:fredholmScalar}
(a) Assume that $r \geq 0$, that $\Im \theta_\Lambda \geq 0$ and that each of the quasiperiodic problems in~\eqref{eq:heAlpha} is uniquely solvable in $\H^1_\talpha(\Omega_H^\Lambda)$. 
If the differences $\| {\bm\zeta}_\s  - {\bm\zeta} \|_{W^{1,\infty}(\{ y \in \R^3: \, \y_3>0\}, \, \R^{3\times 3})}$ and $\| \theta \circ {\bm\zeta}_\s \circ {\bm\zeta}^{-1} - \theta_\Lambda \|_{L^\infty(\Gamma)}$ additionally are small enough, then~\eqref{eq:varFormScal} is uniquely solvable in $L^2(\Wast; \HH^1_\talpha(\Omega_H^\Lambda))$ for all $f_\s \in \HH^{-1/2}(\Gamma)$. 
Further, the variational formulation~\eqref{eq:varFormPertScal} involving the perturbed periodic surface $\Gamma_\s$ is then uniquely solvable in $\HH^1(\Omega_\s^H)$.  

(b) Under the assumptions of part (a), assume further that $f_\s \in \HH^{-1/2}_1(\Gamma)$ satisfies that $\talpha \mapsto  \J_\Omega f_\s(\talpha,\cdot)$ vanishes in a neighborhood of the $\LambdaAst$-periodic set $\{ \bm\xi \in \Wast: \, \exists \, \j \in \Z^2 \text{ s.~th. } |\LambdaAst\j + \bm\xi | = k \}$. 
Then the solution $w_\s$ to~\eqref{eq:varFormScal} belongs to $\HH^1_\0(\Wast; \HH^1_\talpha(\Omega_H^\Lambda))$ and the solution $u$ to the variational formulation~\eqref{eq:varFormPertScal} involving the perturbed periodic surface $\Gamma_\s$ belongs to $\HH^1_1(\Omega_\s^H)$.  
\end{theorem}
\begin{remark}
For scattering of an incident Herglotz wave function $v_g$ from $\Gamma_\s$, the scattered field is described by~\eqref{eq:varFormPertScal} with boundary datum $f = - (\partial v_g/\partial\bm\nu) + \theta v_g$. 
Example~\ref{ex:1} shows that if $g$ is continuously differentiable with compact support contained in the open lower half-sphere, then $f$ satisfies the conditions of Theorem~\ref{th:fredholmScalar}(b), such that the scattered field belongs to $\HH^1_1(\Omega_\s^H)$. 
\end{remark}
\begin{proof}
(a) 
We omit to denote the image space $\R^{3\times 3}$ of $A_\s$ explicitly. 
The assumption that  $\| {\bm\zeta}_\s  - {\bm\zeta} \|_{W^{1,\infty}(\{ 0<\y_3<H\})} < h$ for some $h>0$ implies that   
\begin{align*}
  \|  \x' \mapsto \Phi_\s(\x') -  \x' \|_{W^{1,\infty}(\{ 0<\y_3<H\})} 
  & = \|  \x' \mapsto [{\bm\zeta}- {\bm\zeta}_\s] \circ {\bm\zeta}_\s^{-1}(\x')\|_{W^{1,\infty}(\{ 0<\y_3<H\})} \\
  & \leq h \, \| {\bm\zeta}_\s^{-1} \|_{W^{1,\infty}(\R^3)}. 
  %
\end{align*}
In particular, $\|  \x' \mapsto \nabla (\Phi_\s(\x') -  \x') \|_{L^\infty(\{ 0<\y_3<H\})}\leq h \, \| {\bm\zeta} \|_{W^{1,\infty}(\R^3)}$.
Continuity of the determinant in particular implies that $| \, |\det \nabla \Phi_\s(\x')| - 1| = | \, |\det \nabla \Psi_\s(\x')| - \det \nabla \x'| \leq C h$ for some constant $C>0$. 
Thus, 
\begin{align*}
  \| A_\s - I_3 & \|_{L^\infty(\Omega_H^\Lambda)} 
  = 
  \big\|  \x \mapsto \big[|[\det \nabla_{\x'} \Phi_\s]| \big[ [ \nabla_{\x'} \Phi_\s] [\nabla_{\x'} \Phi_\s]^\top  \big] \big] \circ \Psi_\s(\x) - \x \big\|_{L^\infty(\Omega_H^\Lambda)} \\
  \leq 
  &  \big\| \big[|[\det \nabla_{\x'} \Phi_\s]| \big[ [ \nabla_{\x'} \Phi_\s] [\nabla_{\x'} \Phi_\s]^\top\big]\big](\x') - \Phi_\s(\x') \|_{L^\infty(\{ 0<\y_3<H\})} \| \Psi_\s \|_{L^\infty(\Omega_H^\Lambda)} \\
  \leq 
  & C \big\| |\det \nabla \Phi_\s (\x')| \,\big[ \nabla \Phi_\s(\x')\,\nabla \Phi_\s(\x')^\top \big] - \big[ \nabla \Phi_\s(\x')\,\nabla \Phi_\s(\x')^\top \big] \big\|_{L^\infty(\{ 0<\y_3<H\})}  \\
  & \quad + \| \nabla \Phi_\s(\x')\,\nabla \Phi_\s(\x')^\top   - \Phi_\s(\x')^\top] \|_{L^\infty(\{ 0<\y_3<H\})} 
 \leq C h.    
\end{align*}
The proof that $\| c_\s - 1 \|_{L^\infty(\Omega_H^\Lambda)} \leq C h$ follows by similar estimates. 
Further, 
\begin{align*}
  \| \theta_\s - \theta_\Lambda&  \|_{L^\infty(\Gamma)} 
  = 
  \big\| |[\det \nabla_{\x'} \Phi_\s]|\circ \Psi_\s(\x) \ |[\nabla_{\x'} \Phi_\s]^{-\top} \bm{\nu}| \circ \Psi_\s(\x) \ \theta \circ \Psi_\s(\x) - \theta_\Lambda(\x) \|_{L^\infty(\Gamma)} \\
  & \leq \big\| \x' \mapsto (|\det \nabla \Phi_\s (\x')|-1) |\nabla \Phi_\s'(\x')^{-\top} \bm{\nu}(\x')| \, \theta(\x') \big\|_{L^\infty(\Gamma_\s)} \\
  & \quad + \big\| \x' \mapsto \big[|\nabla \Phi_\s(\x')^{-\top} \bm{\nu}(\x')|-1 \big] \theta(\x') \big\|_{L^\infty(\Gamma_\s)}  
   + \| \theta\circ\Phi_\s - \theta_\Lambda \|_{L^\infty(\Gamma)} \\
  & \leq C h \| \theta \|_{L^\infty(\Gamma_\s)}
  + C \big\| \x \mapsto \nabla (\Phi_\s(\x')^{-1} - \x') \big\|_{L^\infty(\Gamma_\s)} \, \| \theta \big\|_{L^\infty(\Gamma_\s)} \\
  & \quad + \| \theta\circ\Psi_\s - \theta_\Lambda \|_{L^\infty(\Gamma)}
  \leq C h \| \theta \|_{L^\infty(\Gamma_\s)} + \| \theta\circ\Psi_\s - \theta_\Lambda \|_{L^\infty(\Gamma)}.   
\end{align*}
If $\| \theta\circ\Psi_\s - \theta_\Lambda \|_{L^\infty(\Gamma)} \leq h$, then $\theta$ is bounded in norm by $(1+h) \| \theta_\Lambda \|_{L^\infty(\Gamma)}$, such that 
\[
  \| \theta_\s - \theta_\Lambda \|_{L^\infty(\Gamma)}
  \leq C h \| \theta_\Lambda \|_{L^\infty(\Gamma)}.
\]
As $a_{k,\talpha}$ is by assumption invertible (see Theorem~\ref{th:uniqueAlpha}) for each $\talpha$, Theorem~\ref{th:exiSolScalPerio} states that~\eqref{eq:varFormScal} is solvable in $L^2(\Wast; \HH^1_\talpha(\Omega_H^\Lambda))$ when the non-periodic terms involving $A_\s$, $c_\s$ and $\theta$ are omitted. 
All further parts of the sesquilinear form in~\eqref{eq:varFormScal} are perturbations that are uniformly bounded in $\talpha$ and can be made arbitrarily small by choosing $h$ small enough due to the above estimates. 
As the set of invertible operators is an open subset of the space of bounded operators with respect to the operator norm, choosing $h$ small enough is sufficient for the variational problem~\eqref{eq:varFormScal} to possess a unique solution $w_\s \in L^2(\Wast; \HH^1_\talpha(\Omega_H^\Lambda))$ for any $f_\s \in \HH^{-1/2}(\Gamma)$. 
By Theorem~\ref{th:scalPertEqui}, $J_\Omega^{-1} w_\s \in \HH^1(\Omega_\s^H)$ then solves~\eqref{eq:varFormPertScal}.

(b) Computing derivatives of the solution $w_\s(\talpha,\cdot)$ to~\eqref{eq:varFormScal} with respect to $\talpha_{1,2}$ as in the proof of Theorem~\ref{th:exiSolScalPerio} shows that these functions satisfy a variational problem with the sesquilinear form from~\eqref{eq:varFormScal} and the right-hand side $F_{\talpha_{1,2}}$ from~\eqref{eq:RHSDeriAlpha}. 
The proof of Theorem~\ref{th:exiSolScalPerio}(b), see in particular estimate~\ref{eq:aux1146}, hence shows that if the support of $\talpha \mapsto \J_\Gamma w_{1,2}'$ does not include vectors $\bm\xi$ such that $|\LambdaAst \j + \bm\xi| = k$ ensures that $w_{1,2}'$ belongs to $L^2(\Wast; \HH^1_\talpha(\Omega_H^\Lambda))$ (compare Corollary~\ref{th:coconut}). 
\end{proof}



\providecommand{\noopsort}[1]{}

\end{document}